\documentclass[11pt]{article}
\usepackage[utf8]{inputenc}
\usepackage{t1enc}
\usepackage{amsmath, amssymb, amsfonts, amsthm, amsopn,epsfig}
\usepackage[none]{hyphenat}
\usepackage{tikz}
\usepackage{float}
\usepackage{graphicx}
\usepackage{caption}
\usepackage[margin=1in]{geometry}
\usepackage[toc,page]{appendix}
\usepackage{epstopdf}
\usepackage{etoolbox}
\usepackage[colorlinks=true]{hyperref}
\usepackage{dsfont}
\usepackage{verbatim}
\hypersetup{
	colorlinks=true,
	linkcolor=blue,
	filecolor=magenta,      
	urlcolor=cyan,
	citecolor=blue
}
\usepackage{cleveref}
\usetikzlibrary{topaths,calc,matrix, positioning,decorations,calligraphy}
\tikzstyle{vertex}=[circle,draw=black,fill=black,inner sep=0,minimum size=5pt,text=white,font=\footnotesize]

\theoremstyle{plain}
\newtheorem{theorem}{Theorem}
\newtheorem{corollary}[theorem]{Corollary}
\newtheorem{claim}[theorem]{Claim}
\newtheorem{lemma}[theorem]{Lemma}
\newtheorem{conjecture}[theorem]{Conjecture}
\newtheorem{problem}[theorem]{Problem}

\theoremstyle{definition}
\newtheorem{definition}{Definition}

\newcommand*{\calL}{{\mathcal{L}}}

\newcommand*{\sdelta}{{\bar{\delta}}}

\DeclareMathOperator*{\True}{True}
\DeclareMathOperator*{\False}{False}
\DeclareMathOperator*{\sign}{sign}

\DeclareMathOperator*{\tower}{tw}

\title{Ramsey numbers of semi-algebraic and semi-linear hypergraphs}
\author{Zhihan Jin \thanks{ETH Zurich, \emph{e-mail}: \textbf{zhijin@student.ethz.ch}}\footnotemark[1] \and Istv\'an Tomon \thanks{Ume\r{a} University, \emph{e-mail}: \textbf{istvan.tomon@umu.se}}\footnotemark[2]}
\date{}

\begin{document}
	\maketitle
	
	\sloppy
	
	\begin{abstract}
		An $r$-uniform hypergraph $H$ is \emph{semi-algebraic} of complexity $\mathbf{t}=(d,D,m)$ if the vertices of~$H$ correspond to points in $\mathbb{R}^{d}$ and the edges of $H$ are determined by the sign-pattern of $m$  degree-$D$ polynomials. Semi-algebraic hypergraphs of bounded complexity provide a general framework for studying geometrically defined hypergraphs. 
		
		The much-studied \emph{semi-algebraic Ramsey number} $R_{r}^{\mathbf{t}}(s,n)$ denotes the smallest $N$ such that every $r$-uniform semi-algebraic hypergraph of complexity $\mathbf{t}$ on $N$ vertices contains either a clique of size $s$ or an independent set of size $n$. Conlon, Fox, Pach, Sudakov, and Suk
  proved that $R_{r}^{\mathbf{t}}(n,n)<\mbox{tw}_{r-1}(n^{O(1)})$, where $\mbox{tw}_{k}(x)$ is a tower of 2's of height $k$ with an $x$ on the top.	This bound is also the best possible if $\min\{d,D,m\}$ is sufficiently large with respect to $r$. 
		They conjectured that in the asymmetric case, we have $R_{3}^{\mathbf{t}}(s,n)<n^{O(1)}$ for fixed $s$. We refute this conjecture by showing that $R_{3}^{\mathbf{t}}(4,n)>n^{(\log n)^{1/3-o(1)}}$ for some complexity $\mathbf{t}$.

		In addition, motivated by results of Bukh and Matou\v sek and Basit, Chernikov, Starchenko, Tao and Tran, we study the complexity of the  Ramsey problem when  the defining polynomials are linear, that is, when $D=1$. In particular, we prove that $R_{r}^{d,1,m}(n,n)\leq 2^{O(n^{4r^2m^2})}$, while from below, we establish $R^{1,1,1}_{r}(n,n)\geq 2^{\Omega(n^{\lfloor r/2\rfloor-1})}$.

	\end{abstract}
 
\noindent
\textbf{AMS Subject Classification.} 05D10

\noindent
 \textbf{Keywords.} Ramsey theory, hypergraphs, semi-algebraic
	
	\section{Introduction}
	Given positive integers $r,s,n$, the {\em Ramsey number} $R_{r}(s,n)$ denotes the smallest $N$ such that every $r$-uniform hypergraph on $N$ vertices contains either a clique of size $s$ or an independent set of size $n$. For convenience, we write $R_{r}(n)$ instead of $R_{r}(n,n)$. 
	In the case of graphs, that is $r=2$, classical results of Erd\H{o}s and Szekeres \cite{ESz35} and Erd\H{o}s \cite{E47} tell us that $R_{2}(n)=2^{\Theta(n)}$, and in case $s$ is fixed and $n$ is sufficiently large, we have $R_{2}(s,n)=n^{\Theta(s)}$. However, in case $r\geq 3$, the Ramsey numbers are less understood. Erd\H{o}s and Rado \cite{ER52} and Erd\H{o}s, Hajnal, and Rado \cite{EHR65} showed that 
	$$\mbox{tw}_{r-1}(\Omega(n^{2}))< R_{r}(n)<\mbox{tw}_{r}(O(n)).$$ 
	Also, in the asymmetric case, we have $R_{r}(s,n)=\tower_{r-1}(n^{\Theta_{r,s}(1)})$ for $s\geq r+2$ \cite{EHR65,MS18}, and $R_{3}(4,n)=2^{n^{\Theta(1)}}$ \cite{EH62}. Here, $\tower_{k}(x)$ is the \emph{tower function} defined as $\tower_1(x):=x$ and $\tower_{k}(x):=2^{\tower_{k-1}(x)}$. Hence, there is an almost exponential gap between the lower and upper bound for $R_{r}(n)$ in case $r\geq 3$, and it is a major open problem to close this gap. Note that, however, the rough order of the asymmetric Ramsey number $R_{r}(s,n)$ is more understood, at least up to the height of the required tower when $s \ge r + 2$. See \cite{CFS10,MS20} for recent developments.
	
    Yet the situation changes if we restrict our attention to hypergraphs that arise from geometric considerations. 
    To this end, an $r$-uniform hypergraph $H$ is \emph{semi-algebraic of complexity $\mathbf{t}=(d,D,m)$} if the vertices of $H$ can be assigned to points in $\mathbb{R}^{d}$ such that the edges of $H$ are determined by the sign-pattern of $m$ polynomials of degree at most $D$ (see the Preliminaries for a formal definition). Semi-algebraic graphs  and hypergraphs of bounded complexity provide a general model to study certain geometric structures, such as intersection and incidence graphs of geometric objects, order types of point configurations, convex subsets of the plane, and so on. The semi-algebraic Ramsey number $R_{r}^{\mathbf{t}}(s,n)$ denotes the smallest $N$ such that any $r$-uniform semi-algebraic hypergraph of complexity $\mathbf{t}$ on $N$ vertices contains either a clique of size $s$ or an independent set of size $n$. Alon, Pach, Pinchasi, Radoi\v ci\'c, and Sharir \cite{APPRS05} proved that $R_2^{\mathbf{t}}(n)=n^{\Theta(1)}$, which was extended by Conlon, Fox, Pach, Sudakov, and Suk \cite{CFPSS14} to $R_{r}^{\mathbf{t}}(n)=\tower_{r-1}(n^{O(1)})$ for general $r$. In \cite{CFPSS14} and \cite{EMRS14}, matching lower bounds are provided in case the parameters $d,D,m$ are sufficiently large with respect to $r$. Specifically, for every $r\geq 2$, there exists $\mathbf{t}$ such that $R_{r}^{\mathbf{t}}(n)=\tower_{r-1}(n^{\Theta(1)}).$ Here and later, the constants hidden by the $O(.),\Omega(.),\Theta(.)$ notation might depend on $r,\mathbf{t}$ and $s$, unless specified otherwise.
	
	\subsection{Asymmetric Ramsey numbers}
	
	In contrast, asymmetric semi-algebraic Ramsey numbers appear to be more mysterious in case $r\geq 3$. For uniformity $r=3$, in the special subcase $d=1$, it was established in \cite{CFPSS14} that $R_{3}^{\mathbf{t}}(s,n)<2^{(\log n)^{O(1)}}$. 
	Furthermore, if $d\geq 2$, a result of Suk \cite{Suk16} shows that $$R_{3}^{\mathbf{t}}(s,n)<2^{2^{(\log n)^{1/2+o(1)}}}=2^{n^{o(1)}}.$$ 
	However, the best known lower bound constructions provide only polynomial growth, which leads to the natural conjecture that $R_{3}^{\mathbf{t}}(s,n)=n^{O(1)}$, formulated in both \cite{CFPSS14} and \cite{Suk16}. Our first main result refutes this conjecture.
	
	\begin{theorem}\label{thm:construction}
		There exists $\mathbf{t}=(d,D,m)$ such that $$R_{3}^{\mathbf{t}}(4,n)>n^{(\log n)^{1/3-o(1)}}.$$ 
	\end{theorem}
	
	As is often the case with hypergraph Ramsey numbers, any result one gets for uniformity $r=3$ can be used to derive bounds for higher uniformity as well due to the powerful tool known as the \emph{Stepping-up lemma}. We discuss this lemma in more detail later. Therefore, in case $r\geq 4$, the result of Suk \cite{Suk16} implies that $R_{r}^{\mathbf{t}}(s,n)<\tower_{r-1}(2^{(\log n)^{1/2+o(1)}})$. A straightforward implementation of the methods of \cite{CFPSS14} combined with Theorem \ref{thm:construction} imply the following lower bound. We omit the proof.
	
	\begin{corollary}
		For every $r\geq 3$, there exist $s$ and $\mathbf{t}=(d,D,m)$ such that 
		$$
		R_{r}^{\mathbf{t}}(s,n)>{\operatorname{tw}}_{r-1}((\log n)^{4/3-o(1)}).
		$$
	\end{corollary}
	
	\subsection{Semi-linear hypergraphs}
	
	As discussed above, if $d,D,m$ are sufficiently large with respect to $r$, then $R^{d,D,m}_r(n)=\tower_{r-1}(n^{\Omega(1)})$. 
	In the constructions provided by both \cite{CFPSS14} and \cite{EMRS14}, the parameters $d$ and $D$ grow with $r$. 
	In particular, \cite{EMRS14} shows that one can take $d=r-3$ for $r\geq 4$. Furthermore, the Veronese mapping\footnote{A Veronese mapping sends $(x_1,\dots,x_d)\in\mathbb{R}^{d}$ to some point whose coordinates are monomials of $x_1,\dots,x_d$. E.g. $(x_1,x_2,x_3)\mapsto (x_1^{2}x_2,x_2^{2}x_3^{2},x_1x_2x_3,x_3^{3})$.} implies that every $r$-uniform semi-algebraic hypergraph of complexity $(d,D,m)$ is also of complexity $(d',r,m)$ for some $d'$ depending only on $d$ and $D$. However, this raises the question whether the upper bound $R^{d,D,m}_r(n)<\tower_{r-1}(n^{O(1)})$ can be significantly improved if we assume that $d$ or $D$ are small compared to $r$. 
	In support of this, Bukh and Matou\v sek \cite{BM12} showed that if $d=1$, that is, when the vertices of the hypergraph correspond to points on the real line, then any $r$-uniform semi-algebraic hypergraph of complexity $(1,D,m)$ containing no clique or independent set of size~$n$ has at most $2^{2^{O(n)}}$ vertices (in \cite{BM12}, the constant hidden by the $O(.)$ notation might depend on the defining polynomials, but a careful inspection of their proof yields that it can be bounded only by a function of $D,m$ and $r$ as well). Also, this bound is the best possible if $D$ and $m$ are sufficiently large. In this paper, we consider what happens if we bound the parameter $D$ instead, that is, the degrees of the defining polynomials.
	
	A semi-algebraic hypergraph of complexity $(d,D,m)$ is \emph{semi-linear}, if $D=1$, that is, all defining polynomials are linear functions. The study of semi-linear hypergraphs was initiated by Basit, Chernikov, Starchenko, Tao, and Tran \cite{BCSTT21}, who considered these hypergraphs in the setting of Zarankiewicz's problem. There are many extensively studied families of graphs that are semi-linear of bounded complexity, for example intersection graphs of axis-parallel boxes in $\mathbb{R}^{d}$, circle graphs, and shift graphs. Motivated by the large literature (e.g. \cite{AG60,CW20,DM19,EH64,LMPT94}) concerned with the Ramsey properties of such families, Tomon \cite{Tlin} studied the Ramsey properties of semi-linear graphs and showed that $R_{2}^{d,1,m}(s,n)\leq n^{1+o(1)}$ holds for every fixed $s,d$ and $m$. 
	This already shows a behavior unique to semi-linearity, as a construction of Suk and Tomon \cite{ST21} shows that $R_2^{d,2,m}(3,n)=\Omega(n^{4/3})$ for some $d$ and $m$. 
	Tomon \cite{Tlin} also proposed the problem of determining the Ramsey numbers of $r$-uniform semi-linear hypergraphs for $r\geq 3$. Our second main result settles this problem.
	
	\begin{theorem}\label{thm:linear}
		For every triple of positive integers $r,d,m$, there exists $c=c(r,m)>0$ such that 
		$$R_{r}^{d,1,m}(n)\leq 2^{cn^{4r^{2}m^{2}}}.$$
	\end{theorem}
	
	Let us highlight that the bound in Theorem \ref{thm:linear} does not depend on the dimension $d$, only on the uniformity $r$ and the number of polynomials $m$. From below, in case $r\geq 3$, the semi-linear Ramsey number grows at least exponentially, showing that Theorem \ref{thm:linear} is sharp up to the value of $c$ and the exponent $4r^{2}m^{2}$.	Indeed, let $H$ be the 3-uniform hypergraph on vertex set $\{1,\dots,N\}$ in which for $x<y<z$, $\{x,y,z\}$ is an edge if $x+z<2y$. Then $H$ is semi-linear of complexity $(1,1,1)$, and it is easy to show that $\omega(H),\alpha(H)\leq \lceil \log_2 N\rceil+1$. Thus, $R^{1,1,1}_{3}(n)\geq 2^{\Omega(n)}$. We show that even faster growth can be achieved by examining certain more convoluted constructions of higher uniformity.
	
	\begin{theorem}\label{thm:lin_constr}
		For every $r\geq 4$, there exists a constant $c>0$ such that  $$R_{r}^{1,1,1}(n)\geq 2^{cn^{\lfloor r/2\rfloor-1}}.$$
	\end{theorem}
	
	Finally, we remark that a multicolor variant of Theorem \ref{thm:linear} also holds.
	
	\begin{theorem}\label{thm:multicolor}
		Let $r,d,m,p$ be positive integers, then there exists $c>0$ such that the following holds. Let $H$ be the complete $r$-uniform hypergraph on $N$ vertices, whose edges are colored with $p$ colors (every edge receiving possibly more than one color) such that each colorclass is semi-algebraic of complexity $(d,1,m)$. Then $H$ contains a monochromatic clique of size at least $(\log N)^{c}$.
	\end{theorem}

	\textbf{Organization.} Our paper is organized as follows. In the next subsection, we introduce our notation and present the definition of semi-algebraic hypergraphs. In Section \ref{sect:lowerbound}, we present the lower bound construction for the asymmetric semi-algebraic Ramsey number, that is, we prove Theorem \ref{thm:construction}. Then, in Section \ref{sect:semilinear}, we discuss semi-linear hypergraphs, and present the proofs of Theorems \ref{thm:linear}, \ref{thm:lin_constr}, and \ref{thm:multicolor}. We finish our paper with some discussion and open problems in Section \ref{sect:concluding remarks}.
	
	\subsection{Preliminaries} \label{sec:preliminary}
        If $v$ is a vector or a sequence, we use $v(i)$ to denote its $i$-th entry. If $M$ is a matrix, $M(i,j)$ denotes the  entry in the $i$-th row and  $j$-th column.
        
	We use the following standard graph theoretic notation. Given an $r$-uniform hypergraph $H$, $\omega(H)$ denotes the \emph{clique number} of $H$, and $\alpha(H)$ denotes its independence number. Given $U\subset V(H)$, $H[U]$ is the subghypergraph of $H$ induced on $U$.
	
	The \emph{degree} (total degree) of a multivariate polynomial is the maximal sum of exponents in a monomial. 
	A polynomial $p:\mathbb{R}^{d}\mapsto \mathbb{R}$ is \emph{linear} if it has degree 1, or equivalently, if there exist $a_1,\dots,a_d,b\in\mathbb{R}$ such that $p(x_1,\dots,x_d)=b+\sum_{i=1}^{d}a_ix_i$.
	
	An $r$-uniform hypergraph $H$ is \emph{semi-algebraic of complexity $(d,D,m)$} if the following holds. 
	There is an enumeration $v_1,\dots,v_{N}$ of the vertices of $H$, an assignment $v_{i}\mapsto p_{i}$ with $p_{i}\in\mathbb{R}^{d}$ for $i\in [N]$, and $m$ polynomials $f_1,\dots,f_m:(\mathbb{R}^{d})^{r}\mapsto \mathbb{R}$ of degree at most $D$ such that for $1\leq i_1<\dots<i_r\leq N$, whether $\{v_{i_1},\dots,v_{i_r}\}$ is an edge of $H$  depends only on the sign-pattern of $(f_1(p_{i_1}, \dots, p_{i_r}),\dots, f_m(p_{i_1},\dots,p_{i_r}))$.
	More precisely, there is a function $\Phi: \{+, -, 0\}^m \mapsto \{\mbox{True}, \mbox{False}\}$ such that $\{v_{i_1},\dots,v_{i_r}\}$ is an edge if and only if 
	$$\Phi(\mbox{sign}(f_1(p_{i_1}, \dots, p_{i_r})),\dots, \mbox{sign}(f_m(p_{i_1},\dots,p_{i_r}))) = \mbox{True}.$$
	Setting $P=\{p_1,\dots,p_{N}\}$ and $\mathbf{f}=(f_1,\dots,f_m)$, we say that $(P,\mathbf{f})$ is a \emph{witness} for the pair $(H,(d,D,m))$. The hypergraph $H$ is \emph{semi-linear of complexity $(d,m)$} if it is semi-algebraic of complexity $(d,1,m)$.
	
	A hypergraph $H$ is the \emph{Boolean combination} of the hypergraphs $H_1,\dots,H_{k}$, if $V(H_{i})=V(H)$ for $i\in [k]$, $H_{i}$ has the same vertex-ordering as $H$, and there exists a function $\phi:\{\True, \False\}^{k}\mapsto \{\True,\False\}$ such that $e$ is an edge of $H$ if and only if
	$$\phi(\{e\in E(H_1)\},\dots,\{e\in E(H_{k}\}))=\True.$$
	Note that if a hypergraph is semi-algebraic of complexity $(d,D,m)$, then it is the Boolean combination of at most $2m$ semi-algebraic hypergraphs of complexity $(d,D,1)$.
	
	We systematically omit the use of floors and ceilings whenever they are not crucial.
	
	\section{A lower bound for \texorpdfstring{$R_3^{\mathbf{t}}(4, n)$}{TEXT}}\label{sect:lowerbound}
	
	\subsection{The stepping-up lemma\label{subsec:step-up-graph}}
	The stepping-up lemma of Erd\H{o}s and Hajnal (see, e.g., \cite{stepping_up}) is a powerful tool used for the construction of hypergraphs with good Ramsey properties.
	It shows that if one can construct an $r$-uniform hypergraph on $N$ vertices with no clique of size $s$ and no independent set of size $n$, then one can also construct an $(r+1)$-uniform hypergraph on $2^{N}$ vertices with no clique of size $(2s+r-4)$ and no independent set of size $(2n+r-4)$. 
	Unfortunately, it works only when $r\geq 3$. 
	In this section, we consider a variant of the stepping-up lemma, which works for the case $r=2$, albeit resulting in much weaker bounds.
	
	Let us introduce some notation. For two distinct sequence $\alpha,\beta\in\{0,1\}^N$, let 
	$$\delta(\alpha, \beta) := \min\{i: \alpha(i) \neq \beta(i)\} \in [N].$$
	Let $\prec$ be the  lexicographical order over $\{0,1\}^{N}$, i.e.
	$$\alpha\prec\beta \Leftrightarrow \alpha(\delta(\alpha,\beta))< \beta(\delta(\alpha,\beta)).$$
	Let us highlight some basic properties of $\delta(\cdot,\cdot)$ and $\prec$ (see \cite[Section 4.7]{stepping_up} for the proofs).
	\begin{enumerate}
		\item If $\alpha \prec \beta \prec \gamma$, then $\delta(\alpha, \beta) \neq \delta(\beta,\gamma)$.
		\item If $\alpha_1 \prec \dots \prec \alpha_\ell$, then $\delta(\alpha_1,\alpha_\ell) = \min_{i\le\ell-1} \delta(\alpha_i,\alpha_{i+1})$.
		\item If $\alpha_1 \prec \dots \prec \alpha_\ell$, there is a unique $i$ which achieves the minimum of $\delta(\alpha_i,\alpha_{i+1})$.
	\end{enumerate}
	Now we define our notion of the step-up.
	
	\begin{definition}
		Given a graph $G$ on vertex set $[N]$, the \emph{step-up} of $G$ is the 3-uniform hypergraph $H$ on vertex set $\{0,1\}^{N}$ defined as follows. For $\alpha,\beta,\gamma\in \{0,1\}^{N}$ with $\alpha \prec \beta \prec \gamma$, we have $\{\alpha,\beta,\gamma\}\in E(H)$ if and only if $\delta(\alpha,\beta)< \delta(\beta,\gamma)$ and $\{\delta(\alpha,\beta), \delta(\beta,\gamma)\} \in E(G)$.
	\end{definition}
	
	The next lemma gives a relation between the clique and independence number of $G$ and the clique and independence number of the step-up of $G$.
	
	\begin{lemma}\label{lemma:hom}
		Let $G$ be a graph on $N$ vertices, and let $H$ be the step-up of $G$. Then $\omega(H) \leq \omega(G)+1$ and $\alpha(H) \leq N^{\alpha(G)}+1$.
	\end{lemma}
	\begin{proof}
		Let us first prove $\omega(H)\leq \omega(G)+1$. 
		Suppose that $\alpha_1\prec\dots\prec \alpha_{t}$ are the vertices of a clique in $H$ for some $t$, and let $\delta_{i}=\delta(\alpha_{i},\alpha_{i+1})$ for $i\in [t-1]$. 
		Then $\delta_1<\dots<\delta_{t-1}$. 
		Also, note that for every $1\leq i<j\leq t-1$, we have that  $\delta(\alpha_i,\alpha_j)=\delta_i$. 
		Therefore, as $\{\alpha_{i},\alpha_{j},\alpha_{j+1}\}\in E(H)$, we deduce that $\{\delta_{i},\delta_{j}\}\in E(G)$. This implies that $\{\delta_1,\dots,\delta_{t-1}\}$ is a clique in $G$, showing $t\leq \omega(G)+1$.
		
		\medskip
		
		For the other inequality, suppose $\alpha_1\prec\dots \prec\alpha_{t}$ are the vertices of an independent set in $H$, and let $\delta_i = \delta(\alpha_i,\alpha_{i+1})$ for $i \in [t-1]$. 
		For $I\subset [t-1]$, the subsequence $(\delta_{i})_{i\in I}$ is said to be {\em bad} if it is strictly increasing, and $\{\delta_{i}, \delta_{j}\} \notin E(G)$ for any distinct $i,j\in I$. 
		Also, if $I$ is an interval, we call the subsequence $(\delta_{i})_{i\in I}$ a \emph{continuous subsequence}. 
		For positive integers $\ell$ and $m$, we will show by double induction on $m$ and $\ell$ that any continuous subsequence, which contains at most~$m$ distinct values and no bad subsequences of length $\ell$, is of length at most $m^{\ell-1}$. 
		When $\ell = 1$, the sequence can have length 0 since any sequence of length 1 is bad.
		When $\ell = 2$, note that any $(\delta_i,\delta_{i+1})$ with $\delta_i < \delta_{i+1}$ is bad by considering $(\alpha_i,\alpha_{i+1},\alpha_{i+2}) \notin H$. Thus, the sequence can have at most $m=m^{\ell-1}$ elements.
		In case $m=1$, this is also trivial as $\delta_{i}\neq\delta_{i+1}$ for any $i\in [t-1]$, so a continuous subsequence with at least two elements contains at least two distinct values.
		
		Now let us assume that $m\geq 2$ and $\ell \ge 3$. Suppose that $\delta_{a}, \dots, \delta_{b}$ is some continuous subsequence, which contains at most $m$ distinct values and no bad subsequence  of length $\ell$. 
		Then there is a unique $i \in \{a, \dots, b\}$ such that $\delta_i$ is the minimum. 
		But then $\delta_a, \dots, \delta_{i-1}$ is a continuous subsequence which contains at most $m-1$ distinct values, and no bad subsequence of length $\ell$, thus has length at most $(m-1)^{\ell-1}$ by our induction hypothesis. 	Moreover, observe that  $\delta(\alpha_i,\alpha_k)=\delta_i$ for every $i< k\leq b$. Hence, as $\{\alpha_{i}, \alpha_{k}, \alpha_{k+1}\} \notin E(H)$ for every $k\in \{i+1,\dots,b\}$, we deduce  that $\{\delta_i,\delta_{k}\} \not\in E(G)$.
		This implies that $\delta_{i+1},\dots, \delta_{b}$ does not contain a bad subsequence of length~$\ell-1$.
		Otherwise, such a subsequence together with $\delta_i$ is a bad subsequence of length $\ell$. Therefore, $\delta_{i+1},\dots, \delta_{b}$ is a continuous subsequence that contains at most $m-1$ distinct values and no bad subsequences of length $\ell-1$, so it has length at most $(m-1)^{\ell-2}$ by our induction hypothesis. In conclusion, the length of the sequence $\delta_a,\dots, \delta_b$ is 
		$$	b-a+1 \le (m-1)^{\ell-1} + 1 + (m-1)^{\ell-2}	\le (m-1)^{\ell-1} + m^{\ell-1} - (m-1)^{\ell-1}	= m^{\ell-1},$$
		completing the induction step. 
		
		Notice that $\delta_1, \dots, \delta_{t-1}$ contains no bad subsequence of length $\alpha(G)+1$, as a bad subsequence induces an independent set in $G$. Also, it has at most $N$ distinct values, so we derive that  
		$$t-1\le N^{\alpha(G)}.$$
	\end{proof}
	
	\subsection{Stepping up algebraically}
	In this section, we adapt the methods of \cite{CFPSS14} and \cite{EMRS14} to show that stepping up can be done semi-algebraically. More precisely, we show that if $G$ is semi-algebraic of bounded complexity, then the step-up of $G$ is also semi-algebraic of bounded complexity.
	
	Given a finite set $P\subset \mathbb{R}^{d}$ and a sequence of polynomials $\mathbf{f}=(f_1,\dots,f_m)$ with $f_i:\mathbb{R}^{d}\times \mathbb{R}^{d}\mapsto \mathbb{R}$ for $i\in [m]$, we say that $(P,\mathbf{f})$ is \emph{robust} if $f_{i}(p,q)\neq 0$ for every $i\in [m]$ and $p,q\in P$. Furthermore, say that a graph $G$ is \emph{robustly} semi-algebraic of complexity $\mathbf{t}=(d,D,m)$, if there exists a robust witness for $(G,\mathbf{t})$. 
	\begin{lemma}\label{lem:algebraic-stepping-up_robust}
		For every $\mathbf{t}\in \mathbb{N}^{3}$ there exists $\mathbf{t}'\in\mathbb{N}^{3}$ such that the following holds. Let $G$ be a robustly semi-algebraic graph of complexity $\mathbf{t}$. Then the step-up of $G$ is semi-algebraic of complexity $\mathbf{t}'$.
		
	\end{lemma}
	
	\begin{proof}
		Let $\mathbf{t}=(d,D,m)$, and let $P=\{p_1,\dots,p_{N}\}$ and $\mathbf{f}=(f_1,\dots,f_{m})$ such that $(P,\mathbf{f})$ is a robust witness for $(G,\mathbf{t})$. Choose a parameter $0<\varepsilon<\frac{1}{2}$, which we specify later. For $\alpha \in \{0,1\}^N$, set
		$$q_\alpha:= \sum_{i=1}^N \alpha(i)\varepsilon^i\cdot (1,p_i) \in \mathbb{R}^{d+1},$$
		where $(1,p_{i})$ denotes the $(d+1)$-dimensional vector with first coordinate $1$, followed by the coordinates of $p_i$. Notably, the first coordinate of $q_{\alpha}$ is $\sum_{i=1}^N \alpha(i)\varepsilon^i$. As is easy to check, the lexicographic ordering of $\{0,1\}^N$ agrees with the natural ordering of $q_\alpha(1)$, and we simply take $\prec'$ to be this ordering.
		
		Next, we define a mapping $\sdelta:\mathbb{R}^{d+1}\times\mathbb{R}^{d+1}\to \mathbb{R}^d$, which will play the role of the $\delta$ from the definition of step-up. For points $x,y\in \mathbb{R}^{d+1}$, we set
		$$	\sdelta(x,y):= \left(\frac{x(2)-y(2)}{x(1)-y(1)},\frac{x(3)-y(3)}{x(1)-y(1)},\ldots,
		\frac{x(d+1)-y(d+1)}{x(1)-y(1)}\right) \in\mathbb{R}^d.$$
		Note that $\sdelta(x,y)$ is undefined when $x(1)=y(1)$, but this will cause no issues later. 
		An important observation, which is also easy to verify, is that for every $\alpha,\beta\in \{0,1\}^{N}$, $\alpha\neq \beta$, we have
		\begin{equation} \label{eq:limit-point}
			\lim_{\varepsilon\to 0} \sdelta(q_\alpha,q_{\beta})= p_{\delta(\alpha,\beta)}.
		\end{equation}
		By the robustness of $(P,\mathbf{f})$, there exists $\eta>0$ such that for any $i\in [m]$, and $p,q\in P$, if $p',q'\in\mathbb{R}^{d}$ are such that $||p'-p||<\eta$ and $||q'-q||<\eta$, then $f_{i}(p,q)$ has the same sign as $f_{i}(p',q')$. For $i\in [m]$, define the function $g_i:(\mathbb{R}^{d+1})^{3}\mapsto \mathbb{R}$ such that $g_i(x,y,z)=f_{i}(\sdelta(x,y),\sdelta(y,z))$. Assuming $\varepsilon$ is sufficiently small, we have 
		$$\mbox{sign}(g_i(q_{\alpha},q_{\beta},q_{\gamma}))=\mbox{sign}(f_i(p_{\delta(\alpha,\beta)},p_{\delta(\beta,\gamma)}))$$ for $\alpha,\beta,\gamma\in\{0,1\}^{N}$. Note that $g_i$ might not be a polynomial, but $$h_i(x,y,z)=(x(1)-y(1))^{2D}(y(1)-z(1))^{2D}g_{i}(x,y,z)$$ is a polynomial of degree at most $5D$. Moreover, $h_i$ has the same sign as $g_i$ whenever $x(1),y(1),z(1)$ are distinct.
		
		Let $H$ be the step-up of $G$, that is, $H$ is the 3-uniform hypergraph on vertex set $\{0,1\}^{N}$, in which $\{\alpha,\beta,\gamma\} \in E(H)$ for $\alpha\prec \beta\prec \gamma$ if and only if $\delta(\alpha,\beta) < \delta(\beta,\gamma)$ and  $\{\delta(\alpha,\beta),\delta(\beta,\gamma)\}\in E(G).$ Then, writing $Q=\{q_{\alpha}:\alpha\in\{0,1\}^{N}\}$, $\mathbf{h}=(h_1,\dots,h_m)$, and $\mathbf{t'}=(d+1,5D,m)$, $(Q,\mathbf{h})$ is a witness for $(H,\mathbf{t}')$.
	\end{proof} 
	
	Finally, we note that the condition of being robust can be omitted. Indeed, suppose that $(G,(d,D,m))$ is witnessed by $(P,\mathbf{f})$. Choose $\varepsilon>0$ sufficiently small, and let $\mathbf{f}'=(f_{i}+\varepsilon,f_{i}-\varepsilon)_{i=1}^{m}$. Then $(P,\mathbf{f}')$ is a robust witness for $(G,(d,D,2m))$. 
	This is true because if $f_{i}(x,y)\neq 0$ for $x,y\in P$, then assuming $\varepsilon$ is sufficiently small, we have $\mbox{sign}(f_{i}(x,y)+\varepsilon)=\mbox{sign}(f_{i}(x,y))$. Also, the condition $f_{i}(x,y)=0$ can be replaced by $f_{i}(x,y)+\varepsilon>0$ and $f_{i}(x,y)-\varepsilon<0$. Therefore, we get the following.
	
	\begin{lemma}\label{lem:algebraic-stepping-up}
		For every $\mathbf{t}\in \mathbb{N}^{3}$ there exists $\mathbf{t}'\in\mathbb{N}^{3}$ such that the following holds. Let $G$ be a semi-algebraic graph of complexity $\mathbf{t}$. Then the step-up of $G$ is semi-algebraic of complexity $\mathbf{t}'$.
	\end{lemma}
	
	\subsection{Proof of Theorem \ref{thm:construction}}\label{subsec:base-G}
	In order to finish the proof, we just need a triangle-free semi-algebraic graph $G$ of bounded complexity with small independent sets. Such constructions are already available, let us briefly describe one, provided by \cite{ST21}.
	
	Let $P \subset \mathbb{R}^2$ be a set of $m$ points and $\calL$ be a set of $m$ lines in $\mathbb{R}^2$ with $\Omega(m^{4/3})$ incidences. Such configurations exist and are extremal for the celebrated Szemer\'edi-Trotter theorem \cite{SzT83}. Let $\prec$ be an arbitrary total ordering on $P\cup \calL$, and let $G$ be the graph with vertex set $$V(G):= \{(p, \ell): p \in P, \ell \in \calL, p \in \ell\},$$ in which $(p,\ell)$ and  $(p',\ell')$ are  joined by an edge if $p\prec p', \ell\prec\ell'$, and $p\in\ell'$. It is clear that $G$ is semi-algebraic of constant complexity. Moreover, $G$ satisfies the following properties.
	
	\begin{lemma}\label{lemma:graph_constr}\cite{ST21}
		$\omega(G)=2$, $\alpha(G)\leq 2m$, and $|V(G)|\geq c_0m^{4/3}$ for some constant $c_0>0$.
	\end{lemma}
	
	Now we are ready to prove Theorem \ref{thm:construction}. 
	
	\begin{proof}[Proof of Theorem \ref{thm:construction}]
	Let $n$ be a positive integer, and let $m$ be the maximal positive integer satisfying $n\geq (c_0m^{4/3})^{2m}+2$, where $c_0$ is the constant given by Lemma \ref{lemma:graph_constr}. 
	Then $m=\Omega(\log n/\log\log n)$.  Lemma \ref{lemma:graph_constr} tells us that there exists a semi-algebraic graph $G$ of complexity~$\mathbf{t}$ on $N=\lfloor c_0m ^{4/3}\rfloor$ vertices containing no triangles, and satisfying $\alpha(G)\leq 2m$. Here, $\mathbf{t}$ is some constant. Let $H$ be the step-up of $G$. Then $H$ has $2^{N}=n^{\Omega((\log n)^{1/3}/(\log\log n)^{4/3})}$ vertices, and it satisfies $\omega(H)=3$ and $\alpha(H)\leq N^{2m}+1< n  $ by Lemma \ref{lemma:hom}. Finally, $H$ is semi-algebraic of complexity $\mathbf{t}'$ for some constant $\mathbf{t}'$ by Lemma \ref{lem:algebraic-stepping-up}, finishing the proof.
	\end{proof}
	
	\section{Semi-linear hypergraphs}\label{sect:semilinear}
	
	In this section, we prove Theorems \ref{thm:linear}, \ref{thm:lin_constr} and \ref{thm:multicolor}. As most of this section is devoted to the proof of Theorem \ref{thm:linear}, let us give a brief outline of it. 
	
	First of all, we show that every  $r$-uniform semi-linear hypergraph of constant complexity is the Boolean combination of a constant number of very simple-looking semi-linear hypergraphs, which we call primitive semi-linear. A primitive semi-linear hypergraph $H$ on $N$ vertices can be defined with help of a single $r\times N$ real matrix $P$ as follows. For $1\leq q_1<\dots<q_r\leq N$, $\{q_1,\dots,q_r\}$ is an edge of $H$ if and only if $\sum_{i=1}^{r}P(i,q_i)<0$.
	
	We want to show that if $H$ is semi-linear of complexity $(d,m)$ on $N$ vertices, then $H$ contains either a clique or an independent set of size at least $(\log N)^{\Omega_{r,d,m}(1)}$. Let  $P_1,\dots,P_k$ be $r\times N$ matrices, let $H_{\ell}$ be the primitive semi-linear hypergraph associated with $P_{\ell}$ for $\ell\in [k]$, and suppose that $H$ is the Boolean combination of $H_1,\dots,H_k$. First, we show that one can simultaneously select the same subset of $N'=(\log N)^{\Omega_{r,k}(1)}$ columns in the matrices $P_1,\dots,P_k$ such that the resulting submatrices $P_1',\dots,P_k'$ have the following property. For each $\ell\in [k]$ and $i\in [r]$, there exists some $s=s(\ell,i)$ such that shifted row sequence $P'_{\ell}(i,1)+s,\dots,P_{\ell}'(i,N')+s$ has exponential behavior. Let $H_{\ell}'$ be the induced subhypergraph of $H_{\ell}$ corresponding to the submatrix $P_{\ell}'$ for~$\ell\in [k]$. Now given $1\leq q_1<\dots<q_r\leq N'$, whether $\{q_1,\dots,q_r\}$ is an edge of~$H_{\ell}'$ depends essentially on which $i\in [r]$ achieves the maximum of the set $\{|P_{\ell}'(i,q_i)+s|:i\in [r]\}$, by the exponential growth. In such a setting, we can show that there exists some $C\subset [N']$ such that $|C|=(N')^{\Omega_{r,k}(1)}$ and each~$H'_{\ell}$ is either a clique or an independent set on $C$. But then $C$ corresponds to either a clique or an independent set of size $(N')^{\Omega_{r,k}(1)}=(\log N)^{\Omega_{r,k}(1)}$ in $H$, finishing the proof.
	
	\subsection{Decomposing semi-linear hypergraphs}
	A hypergraph $H$ is \emph{vertex-ordered} if there is a fixed enumeration $v_1,\dots,v_{n}$ of its vertices. 
	In what follows, all hypergraphs considered are vertex-ordered, so for simplicity, we choose not to write it; also, subhypergraphs always inherit the vertex-ordering. 
	
	\begin{definition}
		An $r$-uniform hypergraph $H$ with vertex enumeration $v_1,\dots,v_{N}$ is \emph{primitive semi-linear}, if there exists a matrix $P\in \mathbb{R}^{r\times N}$ such that for $1\leq q_1<\dots<q_{r}\leq N$, $\{v_{q_1},\dots,v_{q_r}\}$ is an edge of $H$ if and only if $$\sum_{i=1}^{r}P(i,q_i)<0.$$
		In this case, say that $P$ is a \emph{witness} for $H$.
	\end{definition}
	
	Note that a matrix $P$ is a witness for a unique hypergraph $H$ (up to isomorphism). Also, by adding some positive noise, we can always assume that no row of $P$ contains repeated elements.
	
	\medskip
	
	First, we show that if $H$ is semi-linear of bounded complexity, then $H$ is the Boolean combination of a bounded number of primitive semi-linear hypergraphs.
	Say that a hypergraph $H$ has \emph{primitive complexity $k$} if it is the Boolean combination of $k$ primitive semi-linear hypergraphs. If $H$ is the Boolean combination of primitive semi-linear hypergraphs $H_{1},\dots,H_{k}$ with witnesses $P_1,\dots,P_{k}$, respectively, define the matrix $M\in \mathbb{R}^{rk\times N}$ such that $$M((j-1)r+i,q)=P_{j}(i,q)$$ for $j\in [k], i\in [r], q\in [N]$. That is, $M$ is the matrix we get by stacking $P_1,\dots,P_{k}$ on each other. Say that $M$ is a \emph{witness} for $H$. Note that a matrix $M$ might be a witness of at most $2^{k}$ hypergraphs~$H$, considering all $2^k$ possible Boolean functions.
	
	\begin{lemma}\label{lemma:primitive}
		For positive integers $r,d,m$, every $r$-uniform semi-linear hypergraph of complexity $(d,m)$ has primitive complexity $k=2m$.
	\end{lemma}
	\begin{proof}
		Suppose that $H$ is an $r$-uniform semi-linear hypergraph of complexity $(d,m)$.
		Let $P=\{p_1,\cdots, p_N\} \subset \mathbb{R}^d$ and $\mathbf{f}=(f_1,\cdots,f_m)$ be a witness for $(H, (d,1,m))$. Also, let $\Phi:\{+,-,0\}^m\mapsto\{\mbox{True},\mbox{False}\}$ be a Boolean function such that given $v_{i_1}<\cdots<v_{i_r}$, $\{v_{i_1},\dots,v_{i_r}\}$ is an edge of $H$ if and only if 
		$$\Phi(\mbox{sign}(f_1(\mathbf{p})),\dots, \mbox{sign}(f_m(\mathbf{p}))) = \mbox{True}.$$
		Here, we write $\mathbf{p}=(p_{i_1},\cdots p_{i_r})$ for simplicity.
		Since $\sign(f_i(\mathbf{p}))=0$ is equivalent to $\sign(f_i(\mathbf{p})) \notin \{+,-\}$, we can find a Boolean function $\Phi':\{\mbox{True},\mbox{False}\}^{2m}\mapsto\{\mbox{True},\mbox{False}\}$ such that 
		\begin{align*}
			&\Phi(\mbox{sign}(f_1(\mathbf{p})),\dots, \mbox{sign}(f_m(\mathbf{p})))\\
			&=\Phi'\big(f_1(\mathbf{p})<0, -f_1(\mathbf{p})<0, \cdots, f_m(\mathbf{p})<0, -f_m(\mathbf{p})<0\big).
		\end{align*}
		For $j\in[m]$, let $g_{2j-1}=f_{j}$ and $g_{2j}=-f_{j}$. For $\ell\in [2m]$, let $H_{\ell}$ be the semi-linear hypergraph on vertex set $V(H)$ for which $\{v_{i_1},\dots,v_{i_r}\}$ is an edge if $g_{\ell}(p_{i_1},\dots,p_{i_r}) < 0$. Then clearly $H$ is the Boolean combination of $H_1,\dots,H_{2m}$, so it is sufficient to show that $H_{\ell}$ is primitive semi-linear for $\ell\in [2m]$.
		
		As $g_{\ell}$ is linear, we can write $g_{\ell}(x_1, \cdots, x_r) = \sum_{i=1}^r \langle a_{\ell,i}, x_i\rangle + b_{\ell}$ with suitable vectors $a_{\ell,1},\dots,a_{\ell,r}\in\mathbb{R}^{d}$ and $b_{\ell}\in\mathbb{R}$.
		Define the matrix $P_{\ell} \in \mathbb{R}^{r\times N}$ as $P_{\ell}(i,j) = \langle a_{\ell,i},p_j\rangle + \frac{b_\ell}{r}$.
		Given $1\leq i_1<\dots<i_{r}\leq N$ and $\mathbf{p}=(p_{i_1},\dots,p_{i_r})$, we have $\sum_{a=1}^r P_{\ell}(a,i_{a}) = g_{\ell}(\mathbf{p})$. 
		Thus, $P_{\ell}$ is a witness for $H_{\ell}$, and consequently $H_{\ell}$ is primitive semi-linear. This completes the proof.
	\end{proof}
	
	\subsection{Streamlining matrices}\label{sec: finding convex matrices}
	
	In this section, we are concerned with matrices. Our goal is to show that in every matrix $M\in\mathbb{R}^{q\times N}$ we can find a large submatrix by removing only columns such that each row has some nice properties. 
	
	\begin{definition}
		A matrix $M\in\mathbb{R}^{q\times N}$ is \emph{row-increasing} if each row of $M$ is strictly monotone increasing. \emph{Row-decreasing} is defined in a similar manner. Furthermore, $M$ is \emph{row-monotone} if each row of $M$ is either strictly monotone increasing or decreasing (we allow both increasing and decreasing rows to appear simultaneously in $M$).
	\end{definition}
	
	The following lemma was originally proved by Burkill and Mirsky \cite{BM73} and Kalmanson \cite{K73} (and then reproved a number of times), but due to its simplicity, we also present its proof here.
	
	\begin{lemma}\label{lemma:monotone}
		Let $M\in\mathbb{R}^{q\times N}$ such that no row contains repeated elements. Then $M$ contains a $q\times N'$ sized row-monotone submatrix $M'$, where $N'=N^{1/2^{q}}$. 
	\end{lemma}
	
	\begin{proof}
		We prove this by induction on $q$. 
		If $q=1$ the statement follows from the well known Erd\H{o}s-Szekeres theorem \cite{ESz35}, stating that any sequence of $k\ell+1$ real numbers contains either a monotone increasing subsequence of length $k+1$, or a monotone decreasing subsequence of length $\ell+1$. So let us assume that $q\geq 2$. 
		Then $M$ contains a $q\times N_0$ sized submatrix $M_0$ such that $N_0=N^{1/2^{q-1}}$, and each of the first $(q-1)$ rows of $M_0$ is either monotone increasing or decreasing. 
		Again, by Erd\H{o}s-Szekeres theorem, the last row of $M_0$ contains either an increasing or decreasing subsequence of length at least $N_0^{1/2}$. 
		Let $M'$ be the submatrix of $M_0$ whose columns are given by this subsequence, then $M'$ is row-monotone.
	\end{proof}
	
	\textit{Remark.} The bound in Lemma \ref{lemma:monotone} is also the best possible. Indeed, let $Q=2^q$ and $n=N^{1/Q}$, let $v_1,\dots,v_{Q}$ be an enumeration of the vectors $\{-1,1\}^{q}$, and let $a_1,\dots,a_{N}$ be the enumeration of the elements of $[n]^{Q}$ in the lexicographical order. Define $M\in\mathbb{R}^{q\times N}$, whose $j$-th column vector is 
	$$M(.,j)=\sum_{\ell=1}^{Q}(2n)^{Q-\ell}\cdot a_j(\ell)\cdot v_{\ell} .$$
	Then for $j<j'$, the sign vector of $M(.,j')-M(.,j)$ is $v_{\ell}$, where $\ell$ is the first coordinate $a_j$ and $a_j'$ differ in. Hence, if for some sequence $1\leq j_1<\dots<j_m\leq N$, the submatrix induced by the columns $j_1,\dots,j_m$ is row monotone, then there exists $k\in [Q]$ such that any two vectors among $a_{j_1},\dots,a_{j_m}$ differ in the $k$-th coordinate first, showing that $m\leq n$.
	
	\begin{definition}
		A sequence $x_1,\dots,x_{N}$ of real numbers is a \emph{cup} if for every $1\leq i<j<\ell\leq N$, we have $x_{i}+x_{\ell}\geq 2x_{j}$, and say that it is a \emph{cap} if $x_{i}+x_{\ell}\leq 2x_{j}$. Furthermore, a matrix $M\in \mathbb{R}^{q\times N}$ is \emph{cupcap}, if $M$ is row-monotone, and each row of $M$ is either a cup or a cap.
	\end{definition}

	\begin{lemma}\label{lemma:convex}
		For every $q$ there exists $c=c(q)>0$ such that the following holds. Let $M\in \mathbb{R}^{q\times N}$ be a matrix such that no row contains repeated elements. Then $M$ contains a $q\times N'$ sized cupcap submatrix, where $N'=c(\log N)^{1/q}$.
	\end{lemma}
	
	First, we show that a slightly weaker version of this lemma already follows by using the Ramsey theory of semi-algebraic hypergraphs. More precisely, we use the multicolor version of semi-algebraic Ramsey numbers to show that the exponent $1/q$ can be replaced with some unspecified constant $\delta=\delta(q)>0$.
	
	\begin{lemma}\cite{CFPSS14}\label{lemma:multicolor}
		For every $\mathbf{t}=(d,D,m)$ and positive integer $s$ there exists $\delta>0$ such that the following holds. Let $H$ be the complete $3$-uniform hypegraph on $N$ vertices whose edges are colored with $s$ colors, each colorclass is a semi-algebraic hypergraph of complexity $\mathbf{t}$. Then $H$ contains a monochromatic clique of size at least $(\log N)^{\delta}$. 
	\end{lemma}
	
	\begin{proof}[Proof of weaker version of Lemma \ref{lemma:convex}]
		By Lemma \ref{lemma:monotone}, $M$ contains a $q\times N_0$ sized row-monotone submatrix $M_0$, where $N_0=N^{1/2^{q}}$. Next, we define a coloring of the complete 3-uniform hypergraph~$H$ on vertex set $[N_0]$ with $3^{q}$ colors, in which each colorclass is semi-algebraic of complexity $(q,1,q)$. Given $f\in \{+,-,0\}^{q}$, the edge $\{a,b,c\}$ with $a<b<c$ is colored with color $f$ if $$\sign(M(i,a)-2M(i,b)+M(i,c))=f(i)$$ for $i\in [q]$. By Lemma \ref{lemma:multicolor}, there exists $\beta>0$ depending only on $q$ such that $H$ contains a monochromatic clique $C\subset [N_0]$ of size $(\log N_0)^{\beta}>(\log N)^{\delta}$ for some $\delta=\delta(q)>0$. The submatrix $M'$ of $M_0$, whose columns are indexed by the elements of $C$, is cupcap, thus finishing the proof.
	\end{proof}
	
	We remark that this weaker version also follows from a result of Moshkovitz and Shapira \cite{MS12} on a generalization of the Erd\H{o}s-Szekeres theorem, giving $\delta=1/3^q$. Now let us provide a self-contained proof of the lemma as well.
	
	\begin{proof}[Proof of Lemma \ref{lemma:convex}]
		Let $f_q(s,t)$ denote the smallest $N$ such that any row-increasing $q\times N$ sized matrix $M$ contains a cupcap submatrix $M'$ such that either
		\begin{itemize}
			\item  $M'$ has size $q\times s$ and every row of $M$ is a cap, or 
			\item  $M'$ has size $q\times t$ and at least one row of $M$ is a cup. 
		\end{itemize}
		First, let us show that $f_1(s,t)\leq 2^{s+t}$, that is, every strictly monotone increasing sequence of $2^{s+t}$ real numbers contains either a cap of length $s$, or a cup of length $t$. We will proceed by induction on $s+t$. In case $s+t\leq 5$, the statement is trivial, as each sequence of length at least $2$ contains both a cup and a cap of length $2$. So suppose that $s+t>5$. Let $N=f(s,t)-1$ and let $x_1,\dots,x_{N}$ be a strictly monotone increasing sequence with no cap of size $s$ and no cup of size $t$. Note that if $(x_1+x_N)/2\leq a<b\leq x_{N}$, then $x_1,a,b$ is a cap, so at most $f(s-1,t)-1$ elements of $x_1,\dots,x_{N}$ are at least $(x_1+x_N)/2$. Similarly, there are at most $ f(s,t-1)-1$ elements of $x_1,\dots,x_{N}$ that are at most $(x_1+x_N)/2$, as for any $x_1\leq a<b\leq (x_1+x_N)/2$, $a,b,x_{N}$ is a cup. Hence, $N\leq f(s-1,t)+f(s,t-1)-2\leq 2^{s+t}-2$, concluding the case $q=1$.
		
		Next, we show that for $q\geq 2$, we have 
		\begin{equation}\label{equ:fqst}
			f_q(s,t)\leq 4q f_q(s-1,t)f_{q-1}(t,t).
		\end{equation}
		Let $N=f_q(s,t)-1$, and let $M$ be a $q\times N$ sized row-increasing matrix containing no $q\times s$ sized submatrix in which every row is a cap, and no $q\times t$ sized cupcap submatrix in which at least one row is a cup. After possibly scaling and shifting each row, we may assume that $M(\ell,1)=0$ and $M(\ell,N)=1$ for $\ell\in [q]$. For $\ell\in [q]$ and $k=1,2,\dots$, let $I_{\ell,k}$ denote the set of indices $i$ such that $2^{-k}<M(\ell,i)\leq 2^{-k+1}$. Also, let $J_{\ell,1},\dots,J_{\ell,p_{\ell}}$ be the subsequence of $I_{\ell,1},I_{\ell,2},\dots$ which contains only the nonempty elements. Then $(J_{\ell,k})_{k\in [p_{\ell}]}$ forms a partition of $[N]\setminus\{1\}$ into non-empty intervals. We make two observations.
		
		\begin{description}
			\item[(i)]  $p_{\ell}< 2f_{q-1}(t,t)$ for every $\ell\in [q]$. 
			
			Indeed, suppose this is not the case, and without loss of generality, assume that $p_q\geq 2f_{q-1}(t,t)$. Let $N_0=\lfloor p_q/2\rfloor$, and select an arbitrary element of $J_{q,2k}$ for $1\leq k\leq N_0$. Let $M_0$ be the submatrix of $M$ induced by the columns indexed by the selected elements. Observe that the last row of $M_0$ is a cup. Remove this row, and let the resulting submatrix be $M_1$. As $N_0\geq f_{q-1}(t,t)$, $M_1$ contains a cupcap submatrix of size $(q-1)\times t$, which gives a cupcap submatrix of $M$ of size $q\times t$ with at least one row being a cup.
			
			\item[(ii)] $\sum_{\ell=1}^{q}p_{\ell}>(N-1)/(2f_q(s-1,t))$.
			
			Again, suppose this is not case.
			Then there exists an interval $J\subset [N]\setminus\{1\}$ such that $|J|\geq f_{q}(s-1,t)$, and $J$ is completely contained in some interval among $J_{\ell,1},\dots,J_{\ell,p_{\ell}}$ for every $\ell\in [q]$. This is true as the total number of endpoints of the intervals $J_{\ell,i}$ for $\ell\in [q]$ and $i\in [p_{\ell}]$ is at most $P=2\sum_{\ell=1}^{q}p_{\ell}$, so there will be an interval $J\subset [N]\setminus\{1\}$ of size at least $(N-1)/P\geq f_q(s-1,t)$ containing no such endpoint. Note that for every $\ell\in [q]$ and $i,j\in J$ with $i<j$, we have that $0=M(\ell,1),M(\ell,i),M(\ell,j)$ is a cap. Let $M_0$ be the submatrix of $M$ induced by the columns indexed by the elements of $J$. Then there are two cases. Either $M_0$ contains a $q\times (s-1)$ sized submatrix whose every row is a cap, which then together with the first column of $M$ is a $q\times s$ sized submatrix of $M$ whose every row is a cap. Or, $M_0$ contains a $q\times t$ sized cupcap submatrix with at least one row that is a cup. In both cases, we get a contradiction.
		\end{description}
		Therefore, by comparing (i) and (ii), we deduce that
		$$\frac{N-1}{2f_q(s-1,t)}<\sum_{\ell=1}^q p_{\ell}< 2qf_{q-1}(t,t),$$
		giving the desired inequality (\ref{equ:fqst}). Now we show by induction on $q$ and $s$ that $f_{q}(s,t)\leq 2^{(q+1)st^{q-1}}$ if $q,s,t\geq 2$, or if $q=1$ and $s=t\geq 2$. As we have $f_1(t,t)\leq 2^{2t}$, this holds for $q=1$ and $s=t$. Now suppose that $q\geq 2$. The inequality is also trivially true if $s=2$, so let us assume $s\geq 3$. Then, we have
		$$f_q(s,t)\leq 4q \cdot f_{q-1}(t,t)\cdot f_q(s-1,t)\leq 4q\cdot 2^{qt^{q-1}}\cdot 2^{(q+1)(s-1)t^{q-1}}\leq 2^{(q+1)st^{q-1}}.$$
		
		In conclusion, we showed that any row-increasing matrix $M$ of size $q\times N$ contains a cupcap matrix of size $q\times n$ if $N\geq 2^{(q+1)n^q}$. Clearly, row-increasing can be relaxed to row-monotone, as multiplying every element of a row by $-1$ does not change whether it is a cupcap matrix.

		Let $M\in\mathbb{R}^{q\times N}$. By Lemma \ref{lemma:monotone}, $M$ contains a row-monotone submatrix $M_0$ of size $q\times N^{1/2^{q}}$. 
		But then, according to the previous argument, $M_0$ contains a $q\times N'$ sized cupcap submatrix with $N'\geq (\frac{1}{q+1}\log_2 N^{1/2^{q}})^{1/q}\geq \frac{1}{4}(\log N)^{1/q}$. This finishes the proof.
	\end{proof}
	
	Lemma \ref{lemma:convex} is sharp up to the constant factor $c$ if $q=1$, as witnessed by the matrix $M=(1\ 2 \ \dots\ N)$. 
	However, we are unable to decide if it is sharp for $q\geq 2$ as well, which we leave as an interesting open problem (see the concluding remarks).
	
	\medskip
	
	If $X=(x_1,\dots,x_{N})$ is a sequence and $\delta\in\mathbb{R}$, the \emph{shift of $X$ by $\delta$} is the sequence $X+\delta=(x_1+\delta,\dots,x_{N}+\delta)$. 
	
	\begin{definition}
		
		Given $\Delta>1$ and $\tau\in\{-,+\}\times \{\searrow,\nearrow\}$, a sequence $x_1,\dots,x_{N}$ is \emph{$(\Delta,\tau)$-exponential} if the following holds:
		\begin{enumerate}
			\item if $\tau=(+,\nearrow)$, then $0<\Delta x_i<x_{i+1}$ for $i\in [N-1]$,
			\item if $\tau=(+,\searrow)$, then $0<\Delta x_{i+1}<x_{i}$ for $i\in [N-1]$,
			\item if $\tau=(-,\nearrow)$, then $0<\Delta(-x_i)<(-x_{i+1})$ for $i\in [N-1]$,
			\item if $\tau=(-,\searrow)$, then $0<\Delta(-x_{i+1})<(-x_{i})$ for $i\in [N-1]$.
		\end{enumerate}
		Also, say that a sequence is \emph{$\Delta$-exponential} if it is $(\Delta,\tau)$-exponential for some $\tau\in\{-,+\}\times \{\searrow,\nearrow\}$.
	\end{definition}

	\begin{lemma}\label{lemma:exponential}
		Let $X=(x_i:i\in [N])$ be a sequence of real numbers, $\Delta >  2$, and $z=\lceil \log_2 \Delta \rceil$. Consider the subsequence $X_{0}=(x_{iz}:i\in [\lfloor (N-1)/z\rfloor])$. 
		\begin{enumerate}
			\item[(a)] If $X$ is a monotone increasing cup, then some shift of $X_0$ is $((+,\nearrow),\Delta)$-exponential.
			\item[(b)] If $X$ is a monotone increasing cap, then some shift of $X_0$ is $((-,\searrow),\Delta)$-exponential.
			\item[(c)] If $X$ is a monotone decreasing cup, then some shift of $X_0$ is $((+,\searrow),\Delta)$-exponential.
			\item[(d)] If $X$ is a monotone decreasing cap, then some shift of $X_0$ is $((-,\nearrow),\Delta)$-exponential.
		\end{enumerate}
	\end{lemma}
	\begin{proof}
		First, consider (a), that is, suppose that $X$ is a monotone increasing cup. Then for all $1 \le i < j < k \le N$, we have $x_k - x_j \ge x_j - x_i > 0$.
		Fixing $i = 1$ and $k = j + 1$, we acquire $(x_{j+1} - x_1) \ge  2(x_j - x_1) > 0$ for $1 < j < N$, and then $(x_{(j+1)z} - x_1) \ge 2^z (x_{jz} - x_1) \ge \Delta (x_{jz} - x_1)$.
		This means that $X_0+(-x_1)$ is $((+,\nearrow),\Delta)$-exponential.
		
		The other three cases hold by similar arguments. In (b) and (c), $X_0+ (-x_N)$ suffices, while in (d),  $X_0+ (-x_1)$.
	\end{proof}
	
	\subsection{Integer matrices and domination hypergraphs}

	In this section, we are concerned with integer matrices $P \in \mathbb{Z}^{r \times N}$. Later, these matrices will correspond to the entry-wise logarithms of appropriately scaled and shifted witness matrices of hypergraphs.

        Let $h$ be an integer threshold. 	We associate an edge-colored $r$-uniform hypergraph $H$ with the pair $(P,h)$ as follows. 
	The vertex set of $H$ is $[N]$, and for $1\leq q_1<\dots<q_r\leq N$, $\{q_1,\dots,q_r\}$ is an edge of color 0 if $$\max\{P(i,q_i): i\in [r]\}\leq h,$$ and for $j\in [r]$, it is an edge of color $j$ if $P(j,q_j)\geq h+4$ and $$P(j,q_j)\geq 2+\max\{P(i,q_i): i\in [r]\setminus\{j\}\}.$$ 
	We refer to $H$ as the \emph{domination hypergraph} of $(P,h)$. 
	For convenience, we extend the definition of $H$ to $r=1$ as well, in which case each vertex at most $h$ is colored with 0, each vertex at least $h+4$ is colored with 1, and $h+1,h+2,h+3$ are uncolored.
	See Figure \ref{fig:domination-coloring} for an example of the domination hypergraph.
	
	The main goal of this section is to prove the following lemma.
	
	\begin{lemma}\label{lemma:domination}
		For every positive integer $r$ and $k$, there exists $c=c(r,k)>0$ such that the following holds for every sufficiently large $N$. Let $P_1,\dots,P_k\in \mathbb{Z}^{r\times N}$ be row-monotone matrices, and let $h_1,\dots,h_k$ be integers. Let $H_{i}$ be the domination hypergraph of $(P_i,h_i)$ for $i\in [k]$. Then there exists $C\subset [N]$ such that $|C|\geq cN^{\frac{1}{rk-k+1}}$ and $C$ is a monochromatic clique in $H_{i}$ for every $i\in [k]$.
	\end{lemma}
	
	For our convenience, instead of Lemma \ref{lemma:domination}, we prove the following slightly more general result.
	
	\begin{lemma}\label{lemma:domination2}
		For every positive integer $k$ and every $k$-tuple of positive integers $\mathbf{r}=(r_1, \cdots, r_k)$, there exists $c=c(k,\mathbf{r})>0$ such that the following holds for every sufficiently large $N$. 
		For each $i \in [k]$, let $P_i \in \mathbb{Z}^{r_i \times N}$ be a row-monotone matrix, and let $h_i$ be an integer threshold.
		Then there exists $C\subset [N]$ such that $C$ is a monochromatic clique in the domination hypergraph $H_{i}$ of $(P_i,h_i)$ for every $i\in[k]$, and $$|C|\geq c N^{\frac{1}{R-k+1}},$$
		where $R=\sum_{i=1}^{k}r_i$.
	\end{lemma}
	
	\begin{proof}
		For simplicity, we first prove the statement with the following additional assumption.
		Then, we remove this assumption and consider the general case in the end of the proof.
		
		\begin{enumerate}
			\item[($*$)] For every $\ell \in [k]$ and $q \in [N]$, there exists an element in the column $P_\ell(\cdot, q)$ that is equal to or  larger than $h_\ell + 4$. 
		\end{enumerate}
		
		We will prove by induction on $R$ that there exists $C\subset [N]$ such that $|C|\geq \frac{1}{3^{R}}N^{\frac{1}{R-k+1}}$ and $C$ is a monochromatic clique in $H_1,\dots,H_k$. 
		In the base case when $R=1$, $C=[N]$ is a clique of color~1 in~$H_1$ by ($*$), and $|C|\geq \frac{1}{3^{R}}N$.

		Now let us assume that $R \geq 2$. If $r_{\ell}=1$ for some $\ell\in [k]$, then again by ($*$), every vertex of $H_{\ell}$ is of color 1. By our induction hypothesis, there exists $C\subset [N]$ of size at least $\frac{1}{3^{R-1}}N^{\frac{1}{(R-1)-(k-1)+1}}>\frac{1}{3^{R}}N^{\frac{1}{R-k+1}}$ such that $C$ is a monochromatic clique in $H_{t}$ for $t\in [k]\setminus \{\ell\}$, but then $C$ is also a monochromatic clique in $H_{\ell}$. Hence, we may assume that $r_1,\dots,r_k\geq 2$.

		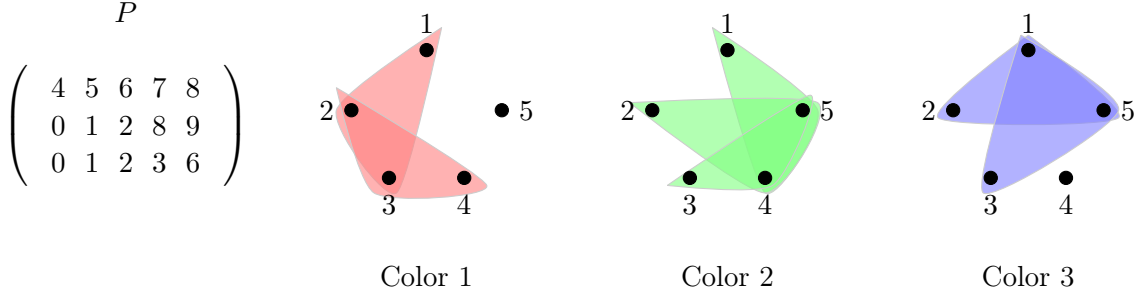
\begin{figure}
			\centering
			\begin{tikzpicture}
				\matrix (m) at (1,0) [matrix of math nodes,left delimiter=(,right delimiter=),row sep=0cm,column sep=0cm] (m) {
					4 & 5 & 6 & 7 & 8 \\
					0 & 1 & 2 & 8 & 9 \\
					0 & 1 & 2 & 3 & 6  \\};
				
				\node at (1,1.5) {$P$};
				
				\node (v1) at (5,1) {};
				\node (v2) at (4,0.2) {};
				\node (v3) at (4.5,-0.7) {};
				\node (v4) at (5.5,-0.7) {};
				\node (v5) at (6,0.2) {};
				
				\node at (5,-2) {Color 1};
				
				\begin{scope}[fill opacity=0.6] 
					\filldraw[color=black!20,fill=red!50] plot [smooth, tension=0.4] coordinates{
						($(v1)+(0.2,0.3)$)
						($(v2) + (-0.2,0)$)
						($(v3) + (0.1,-0.2)$)
						($(v1)+(0.2,0.3)$)};
					
					\filldraw[color=black!20,fill=red!50] plot [smooth, tension=0.4] coordinates{
						($(v2)+(-0.2,0.3)$)
						($(v3) + (-0.2,-0.2)$)
						($(v4) + (0.3,-0.1)$)
						($(v2) + (-0.2,0.3)$)};
				\end{scope}
				
				\node[vertex,label=above:$1$] at (v1) {};
				\node[vertex,label=left:$2$] at (v2) {};
				\node[vertex,label=below:$3$] at (v3) {};
				\node[vertex,label=below:$4$]  at (v4) {};
				\node[vertex,label=right:$5$] at (v5) {};
				
				\node (x1) at (9,1) {};
				\node (x2) at (8,0.2) {};
				\node (x3) at (8.5,-0.7) {};
				\node (x4) at (9.5,-0.7) {};
				\node (x5) at (10,0.2) {};
				
				\node at (9,-2) {Color 2};
				
				\begin{scope}[fill opacity=0.6] 
					\filldraw[color=black!20,fill=green!50] plot [smooth, tension=0.4] coordinates{
						($(x1)+(-0.2,0.3)$)
						($(x4) + (0,-0.2)$)
						($(x5) + (0.2,0)$)
						($(x1)+(-0.2,0.3)$)};
					
					\filldraw[color=black!20,fill=green!50] plot [smooth, tension=0.4] coordinates{
						($(x2)+(-0.3,0.1)$)
						($(x4) + (0,-0.2)$)
						($(x5) + (0.2,0.1)$)
						($(x2) + (-0.3,0.1)$)};
					
					\filldraw[color=black!20,fill=green!50] plot [smooth, tension=0.4] coordinates{
						($(x3)+(-0.3,-0.1)$)
						($(x4) + (0.1,-0.1)$)
						($(x5) + (0.1,0.2)$)
						($(x3) + (-0.3,-0.1)$)};
				\end{scope}
				
				\node[vertex,label=above:$1$] at (x1) {};
				\node[vertex,label=left:$2$] at (x2) {};
				\node[vertex,label=below:$3$] at (x3) {};
				\node[vertex,label=below:$4$]  at (x4) {};
				\node[vertex,label=right:$5$] at (x5) {};
				
				\node (y1) at (13,1) {};
				\node (y2) at (12,0.2) {};
				\node (y3) at (12.5,-0.7) {};
				\node (y4) at (13.5,-0.7) {};
				\node (y5) at (14,0.2) {};
				
				\node at (13,-2) {Color 3};
				
				\begin{scope}[fill opacity=0.6] 
					\filldraw[color=black!20,fill=blue!50] plot [smooth, tension=0.4] coordinates{
						($(y1)+(0,0.2)$)
						($(y2) + (-0.2,-0.1)$)
						($(y5) + (0.2,-0.1)$)
						($(y1)+(0,0.2)$)};
					
					\filldraw[color=black!20,fill=blue!50] plot [smooth, tension=0.4] coordinates{
						($(y1)+(-0.1,0.2)$)
						($(y3) + (-0.1,-0.2)$)
						($(y5) + (0.2,0)$)
						($(y1) + (-0.1,0.2)$)};
					
				\end{scope}
				
				\node[vertex,label=above:$1$] at (y1) {};
				\node[vertex,label=left:$2$] at (y2) {};
				\node[vertex,label=below:$3$] at (y3) {};
				\node[vertex,label=below:$4$]  at (y4) {};
				\node[vertex,label=right:$5$] at (y5) {};
				
			\end{tikzpicture}
			\caption{An example of a domination hypergraph, where we set $h=0$, thus no edges of color 0.}
			\label{fig:domination-coloring}
		\end{figure}

		Next, consider the case when there exists some $\ell\in[k]$ such that $P_{\ell}$ is neither row-increasing or row-decreasing. Without loss of generality, suppose that $\ell=k$. Let $I \subset [r_{k}]$ be the set of indices $i$ such that the $i$-th row of $P_k$ is monotone increasing. By assumption, $I \neq \emptyset$ and $I \neq [r_k]$. For each $q\in [N]$, let $m_q=\max_{i\in I} P_k(i,q)$ and $m'_q=\max_{i\in [r_k]\setminus I} P_k(i,q)$.  Then the sequence $(m_q)_{q\in [N]}$ is monotone increasing while $(m'_q)_{q \in [N]}$ is monotone decreasing. Let $q_0$ be the smallest integer such that $m_{q_0}>m'_{q_0}$, and if there exists no such $q_0$, let $q_0=N$. 
		If $q_0>N/2$, let $P_t'$ be the submatrix of $P_t$ formed by the first $N'=q_0-2$ columns for all $t \in [k]$, and let $H_t'$ be the domination hypergraph of $(P_t',h_t)$. 
		Clearly, $H_{t}'$ is an induced subhypergraph of $H_t$ inheriting the coloring.  
		Note that for every $1\leq q_1<\dots<q_{r_k}\leq q_0-2$, the maximum of $\{P_k'(i,q_i): i\in [r_k] \setminus I\}$ is larger than or equal to 2 plus the maximum of $\{P_k'(i,q_i):i\in I\}$. 
		Remove the rows of $P_k'$ indexed by the elements of $I$, and let $P_{k}''$ be the resulting matrix. If $H_{k}''$ is the domination hypergraph of $(P_{k}'',h_{k})$, then any monochromatic clique in $H_{k}''$ is also a monochromatic clique in $H_{k}'$. We can apply our induction hypothesis to the matrices $P_1',\dots,P_{k-1}',P_{k}''$, as their total number of rows is $R-|I|<R$: we get that there exists $C\subset [N']$ such that $C$ is monochromatic in $H_1',\dots,H_{k-1}',H_{k}''$ and $|C|\geq \frac{1}{3^{R-|I|}}N'^{\frac{1}{R-|I|-k+1}}\geq \frac{1}{3^{R}}N^{\frac{1}{R-k+1}}$. 
		In the other case, when $q_0 \le N / 2$, we proceed similarly, the main difference is that we define $P_t'$ to be the submatrix of $P_t$ formed by the last $N-q_0$ columns of $P_t$ for $t\in [k]$, and remove the rows of $P_k'$ indexed by the elements of $[r_k] \setminus I$.
		
		In the remainder of the proof, we assume that $P_{t}$ is either row-increasing or row-decreasing for every $t\in [k]$. 
		This, plus ($*$), indicates that there is no edge of color 0 in any $H_\ell$.
		
		Let $L \subset [k]$ be the set of indices $i$ such that $P_i$ is row-increasing.
		For each $\ell \in [k]$ and $q \in [N]$, define $s_{\ell,q}$ as follows: 
		\begin{itemize}
			\item if $\ell \in L$, then $s_{\ell,q}$ is the {\em smallest} integer such that for some $i\in [r_\ell-1]$, $P_\ell(i,s_{\ell,q})\geq P_\ell(r_\ell,q)$, with the convention to set $s_{\ell,q}=N$ if no such integer exists;
			\item if $\ell \notin L$, then $s_{\ell,q}$ is the {\em largest} integer such that for some $i = 2, \dots, r_\ell$, $P_\ell(i,s_{\ell,q})\geq P_\ell(1,q)$, with the convention to set $s_{\ell,q}=1$ if no such integer exists.
		\end{itemize}
		Then, consider the following three cases.
		\begin{description}
			\item[Case 1.] \emph{There exists some $\ell \in L$ and $q\in [N]$ such that $q-s_{\ell,q}>N^{\frac{R-k}{R-k+1}}$.}
			
			Without loss of generality, assume that $\ell=k$ and $i < r_k$ such that $P_k(i,s_{k,q}) \ge P_k(r_k,q)$.
			For each $t \in [k]$, let $P_t'$ be the submatrix of $P_t$ in which we keep the columns indexed from $s_{k,q}+2$ to $q$, and let $H'_t$ be the domination hypergraph of $(P_t',h_t)$. 
			Let $N' = q-s_{k,q}-1$ be the number of columns of $P_t'$. 
			Note that for all $q_1', q_2' \in [N']$,
			\begin{equation} \label{eq:crucial-obs}
				\begin{aligned}
					P_k'(r_k,q_1') \le P_k'(r_k,N') = P_k(r_k,q) &\le P_k(i,s_{k,q})  \\
					&\le P_k(i,s_{k,q}+2)-2 = P_k'(i,1)-2 \le P_k'(i,q_2')-2.
				\end{aligned}
			\end{equation}
			
			Let $P_k''$ be the submatrix of $P_k'$ we get by removing row $r_k$, and denote by $H_{k}''$ the domination hypergraph of $(P_k'',h_k)$.	
			The crucial observation here is that a monochromatic clique in $H_k''$ is also a monochromatic clique in $H_k'$. 
			Indeed, if for $1\leq q_1<\dots<q_{r-1}\leq N'$, $\{q_1,\dots,q_{r-1}\}$ is an edge of $H_k''$ of color $j \in [r_k-1]$, then by (\ref{eq:crucial-obs}), $\{q_1,\dots,q_{r-1},q_r\}$ is an edge of $H_k'$ of color~$j$ for every $q_{r-1}<q_r\leq N'$.
			
			The total number of rows of $P_1',\dots,P_{k-1}',P_{k}''$ is $R-1$, so by our induction hypothesis, there exists a set $C \subset [N']$ such that $|C| \ge \frac{1}{3^{R-1}}N'^{\frac{1}{R-k}}>\frac{1}{3^{R-1}}N^{\frac{1}{R-k+1}}$, and $C$ is a monochromatic clique $H_1',\dots,H_{k-1}',H_{k}''$, finishing this case.
			
			\item[Case 2.]  \emph{There exists some $\ell \in [k] \setminus L$ and $q\in [N]$ such that $s_{\ell,q}-q>N^{\frac{R-k}{R-k+1}}$.}
			
			This case can be handled similarly as the former with the following differences: we set $P_t'$ to be the submatrix of $P_t$ with columns indexed from $q$ to $s_{\ell,q}-2$, and remove the first row of~$P_\ell'$.
			
			\item[Case 3.] \emph{For all $\ell \in L, q\in [N]$, we have $q-s_{\ell,q}\leq N^{\frac{R-k}{R-k+1}}$; and for all $\ell \in [k] \setminus L, q \in [N]$, we have that $s_{\ell,q}-q\leq N^{\frac{R-k}{R-k+1}}$.}
			
			Let $z=\left\lceil N^{\frac{R-k}{R-k+1}}\right\rceil+2$, then $C=\{iz:i=2,\dots, \lfloor N'/z\rfloor -1 \}$ is a monochromatic clique in $H_{\ell}$ of color $r_\ell$ if $\ell \in L$, and $C$ is monochromatic of color $1$ in $H_{\ell}$ if $\ell \in [k] \setminus L$.
			
			Indeed, let $\ell \in L$, $2\leq p_1<p_2<\dots<p_r\leq N'/z-1$ and let $q_i=p_iz\in C$ for $i\in [r]$. Then for $i\in [r_\ell-1]$, we have 
			$$P_\ell(i,q_i)\leq P_\ell(i,q_{r_\ell}-z)\leq P_\ell(i,s_{\ell,q_{r_{\ell}}-1}-1)\leq P_\ell(r_\ell,q_{r_\ell}-1)-1\leq P_\ell(r_\ell,q_{r_\ell})-2,$$
			which shows that $\{q_1,\dots,q_{r_\ell}\}$ is indeed an edge of color $r_{\ell}$ in $H_{\ell}$.
			
			Similarly, let $\ell \in [k]\setminus L$, $2\leq p_1<p_2<\dots<p_{r_{\ell}}\leq N'/z-1$ and let $q_i=p_iz\in C$ for $i\in [r]$. Then for $i = 2, \dots, r_\ell$, we have 
			$$P_\ell(i,q_i)\leq P_\ell(i,q_1+z)\leq P_\ell(i,s_{\ell,q_1+1}+1)\leq P_\ell(1,q_1+1)-1\leq P_\ell(1,q_1)-2,$$
			which shows that $\{q_1,\dots,q_{r_\ell}\}$ is indeed an edge of color $1$ in $H_{\ell}$.
			
			As $|C|\geq N/(N^{\frac{R-k}{R-k+1}}+3)-3> \frac{1}{3^{R}}N^{\frac{1}{R-k+1}}$ (if $N$ is sufficiently large), we are done in this case as well.
		\end{description}
		
		Now let us assume that ($*$) does not necessarily  hold. 
		Say that an index $\ell\in [k]$ is \emph{bad} if there exists $q\in [N]$ such that all elements of the column $P_{\ell}(.,q)$ are smaller than $h_\ell + 4$. 
		Let $b$ be the number of bad indices. 
		Note that $0 \le b \le k$.
		We show that there exists $C\subset [N]$ such that $C$ is a monochromatic clique in $H_1,\dots,H_k$ and $|C|\geq \frac{1}{3^{R+b}}N^{\frac{1}{R-k+1}}$. 
		We proceed by induction on $b$. 
		The base case $b=0$ was proved above, so let us assume that $b\geq 1$. 
		
		Without loss of generality, we can assume that $k$ is bad. Let $S \subset [N]$ be the set of indices $q$ such that all elements of the column $P_k(\cdot,q)$ is smaller than $h_k+4$. If $|S| \ge \frac{N}{2}$, set $P_t'$ to be the submatrix of $P_t$ formed by the columns indexed by $S$, for all $t \in [k]$. Then, all elements of $P_k'$ are bounded above by $h_k+3$. Furthermore, delete the first three columns and the last three columns of each matrix $P_t'$, resulting in the matrix $P_t''$. Let $H_t''$ be the domination hypergraph of $(P_t'',h_t)$, and let $N''\geq N/2-6$ be the the number of columns of the matrix $P_t''$. Then, every element of $P_k''$ is upper bounded by $h_k$, i.e. $H_k''$ itself is a monochromatic clique of color 0. By our induction hypothesis, there exists $C \subset [N'']$ such that $C$ is a monochromatic clique in $H_1'', \dots, H_{k-1}''$ and $$|C| \ge \frac{1}{3^{R-r_k+b-1}}N''^{\frac{1}{R-r_k-(k-1)+1}} \ge \frac{1}{3^{R+b}} N^{\frac{1}{R-k+1}}.$$
		So in this case, we are done.
		If $|S| < \frac{N}{2}$, set $P_t'$ to be the submatrix of $P_t$ formed by the columns indexed by $[N]\setminus S$, and let $H_t'$ be the corresponding domination hypergraph of $(P_t',h_t)$ for all $t \in [k]$.
		Let $N' = N-|S|$.
		Then, for every $q \in [N']$, there exists some element of the column $P_k'(\cdot, q)$ larger than or equal to $h_k + 4$, so the number of bad indices with respect to the new matrices $P_1',\dots,P_k'$ is at most $b-1$. Hence, by our induction hypothesis, there exists $C\subset [N']$ such that $C$ is a monochromatic clique in $H_1',\dots,H_k'$, and $|C|\geq \frac{1}{3^{R+b-1}}N'\geq \frac{1}{3^{R+b}}N$. This finishes the proof.
	\end{proof}
	
	We finish this section by showing that the exponent in Lemma \ref{lemma:domination2} is the best possible for every~$k$ and every~$k$-tuple of positive integers $\mathbf{r} = (r_1, \dots, r_k)$ up to the value of $c$.
	
	\begin{claim}
		Let $k$ and $r_1,\dots,r_k$ be positive integers. 
		Then for infinitely many $N$, there exist $k$ row-increasing matrices  $P_1\in\mathbb{Z}^{r_1\times N},\dots,P_k\in\mathbb{Z}^{r_k\times N}$ and integer thresholds $h_1,\dots,h_k$ such that if~$C$ is a monochromatic clique in the domination hypergraph $H_i$ of $(P_i,h_i)$ for $i\in [k]$, then $$|C|<R\cdot N^{\frac{1}{R-k+1}}+2R,$$
		where $R=\sum_{i=1}^{k}r_i$.
	\end{claim}
	
	\begin{proof}
		We may assume that $r_\ell \ge 2$ for all $\ell$ as adding matrices with a single row only increases the upper bound. Let $n$ be a positive integer and set $N=n^{R-k+1}$. Also, for $\ell=1,\dots,k$, let $\tau(\ell)=-\ell+\sum_{t<\ell}r_{t}$, and for $i\in [r_{\ell}]$ and $q\in [N]$, define 
		$$P_{\ell}(i,q)=q-n^{\tau(\ell)+i}.$$
		We set each threshold to be $h_{\ell}=-N-4$, so the domination hypergraph $H_{\ell}$ of $(P_{\ell},h_{\ell})$ has no edge of color $0$. Let $1\leq a_1< \dots<a_m$ be such that $C=\{a_1,\dots,a_m\}$ is a monochromatic clique in $H_{\ell}$ of some color $c_{\ell}\in [r_{\ell}]$ for $\ell\in [k]$. We show that $|C|\leq Rn+2R$. Note the following.
		
		\begin{itemize}
			\item[(a)] If $c_{\ell}>1$, $C$ cannot contain $r_{\ell}\leq R$ elements in an interval of length $n^{\tau(\ell)+c_{\ell}}$. Indeed, such a set would form an edge of color less than $c_{\ell}$ in $H_{\ell}$.
			
			\item[(b)] If $c_{\ell}<r_{\ell}$ and $m \ge 2R$, then $a_{m-R+1}-a_{R}\leq n^{\tau(\ell)+c_{\ell}+1}$. 
			Otherwise, $$\{a_1,\dots,a_{c_\ell},a_{m-r_\ell+c_\ell+1},\cdots, a_m\}$$ is an edge of color at least $c_{\ell}+1$ in $H_{\ell}$.
		\end{itemize}

		By (a) and (b), if $1<c_{\ell}<r_{\ell}$ for some $\ell\in [k]$, then we must have $|C|-2R<R n$, which implies $|C|\leq Rn+2R$. Therefore, assume that for every $\ell\in [k]$, we have $c_{\ell}\in \{1,r_{\ell}\}$. 
		Let $j\in [k]$ be the smallest integer such that $c_j=1$, and set $j=k+1$ if there is no such index. 
		If $j=1$, then $a_{m-R+1}-a_{R}\leq n$ by (b) applied to $\ell=1$, so $|C|\leq n+2R$. 
		If $j=k+1$, then no $r_k$ elements are contained in an interval of length $n^{R-k}$ by (a) applied to $\ell=k$, so $|C|\leq RN/n^{R-k}=Rn$. Otherwise, we can apply (b) with $\ell=j$ and (a) with $\ell=j-1$ to get that $a_{m-R+1}-a_R\leq n^{\tau(j)+2}$, and no interval of length $n^{\tau(j-1)+r_j}=n^{\tau(j)+1}$ contains more than $R$ elements of $C$. From this, we get $|C|\leq Rn+2R$.
	\end{proof}
	
	\subsection{Proof of Theorem \ref{thm:linear}}
	
	In this section, we finish the proof of Theorem \ref{thm:linear}. Given a matrix $M$ with $kr$ rows and $\ell\in [k]$, the \emph{$\ell$th block of $M$} is the submatrix of $M$ with $r$ rows indexed $(\ell-1)r+1,\dots,\ell r$.
	
	\begin{proof}[Proof of Theorem \ref{thm:linear}]
		Let $H$ be an $r$-uniform semi-linear hypergraph of complexity $(d,m)$ on $N$ vertices. By Lemma \ref{lemma:primitive}, there exists $k=2m$ such that $H$ has primitive complexity $k$. Let $M\in \mathbb{R}^{rk\times N}$ be a matrix witnessing $H$, and for $\ell=1,\dots,k$, let $H_{\ell}$ be the primitive semi-linear hypergraph witnessed by the $\ell$-th block of $M$ with $V(H_\ell)=V(H)$.  Then $H$ is the Boolean combination of $H_1,\dots,H_k$.
		
		By Lemma \ref{lemma:convex}, $M$ contains an $rk\times N_0$ sized cupcap submatrix $M_0$, where $N_{0}=c_0(\log N)^{1/(rk)}$ for some $c_0=c_0(r,k)>0$.
		Now set $\Delta=2r$, $z=\lceil \log_2 \Delta\rceil$ and $N_1=\lfloor (N_0-1)/z\rfloor$. Let $M_1$ be the $rk\times N_1$ sized submatrix of $M_0$ in which we keep the columns indexed by $z,2z,\dots,N_1z$. Then by Lemma \ref{lemma:exponential}, for each row $i\in [rk]$ there exist $s_i\in\mathbb{R}$ such that the shifted sequence of the $i$-th row, that is, $(M_1(i,j)+s_i)_{j\in [N_1]}$ is $\Delta$-exponential. Note that $M_1$ is a witness for some subhypergraph $H'$ of $H$. For simplicity, assume that the vertex set of $H'$ is $[N_1]$ with the natural ordering. Also, for $\ell\in [k]$, let $P_{\ell}'$ be the $\ell$-th block of $M_1$, and let $H_{\ell}'$ be the hypergraph  on vertex set $[N_1]$ witnessed by $P_{\ell}'$. Then $H'$ is a Boolean combination of $H_1',\dots,H_{k}'$.
		
		For $\ell\in [k]$, let $Q_{\ell}$ be the $r\times N_1$ sized matrix defined as $Q_{\ell}(i,j)=P_{\ell}'(i,j)+s_{(\ell-1)r+i}$, and set $S_{\ell}=\sum_{i=(\ell-1)r+1}^{\ell r}s_i$. 
		Furthermore, define the integer matrices $L_{\ell}\in\mathbb{Z}^{r\times N_1}$ such that $L_{\ell}(i,j)=\lfloor\log_{\Delta} |Q_{\ell}(i,j)|\rfloor$, and set $h_{\ell}=\lfloor \log_{\Delta}|S_\ell|\rfloor-2$ with convention to set $h_\ell=-\infty$ if $S_\ell=0$, which will cause no problem in further discussion.
		
		First of all, observe that as every row of $Q_{\ell}$ is $\Delta$-exponential, $L_{\ell}$ is row-monotone. Let $K_{\ell}$ be the domination hypergraph of $(L_{\ell},h_{\ell})$. We show that if $e\in K_{\ell}$, then whether $e$ is an edge of $H_{\ell}'$ is determined by the color of $e$ in $K_{\ell}$. By definition, if $i\in [r]$ and $j\in [N_1]$, we have 
		$$\Delta^{L_{\ell}(i,j)}\leq |Q_{\ell}(i,j)|< \Delta^{L_{\ell}(i,j)+1}.$$
		Write $e=\{q_1,\dots,q_r\}$ with $1\leq q_1<\dots<q_{r}\leq N_1$.  We have $e\in H'_{\ell}$ if and only if 
		$$0>\sum_{i=1}^{r}P'_{\ell}(i,q_i)=-S_{\ell}+\sum_{i=1}^{r}Q_{\ell}(i,q_i).$$
		If $e$ has color 0 in $K_{\ell}$, then $L_{\ell}(1,q_1),\dots,L_{\ell}(r,q_r)$ are at most $h_{\ell}$. Hence,
		$$\left|\sum_{i=1}^{r}Q_{\ell}(i,q_i)\right|< \sum_{i=1}^{r}\Delta^{L_{\ell}(i,j)+1}\leq r\cdot  \Delta^{h_{\ell}+1}<|S_{\ell}|.$$
		Therefore, $e\in E(H_{\ell}')$ if and only if $S_{\ell}>0$. Now suppose that $e$ has color $j$ for some $j\in [r]$ in $K_{\ell}$. 
		Then $L_{\ell}(j,q_{j})\geq 2+L_{\ell}(i,q_{i})$ for every $i\in [r]\setminus \{j\}$ and $L_{\ell}(j,q_{j})\geq h_{\ell}+4$. But then
		$$\left|-S_{\ell}+\sum_{i\in [r]\setminus \{j\}}Q_{\ell}(i,q_i)\right|< \Delta^{h_{\ell}+3}+ \sum_{i\in [r]\setminus \{j\}}\Delta^{L_{\ell}(i,q_i)+1}\leq r\cdot \Delta^{L_{\ell}(j,q_j)-1}<|Q_{\ell}(i,j)|. $$
		Therefore, $e\in E(H'_{\ell})$ if and only if $Q_{\ell}(i,j)<0$. But the $j$-th row of $Q_{\ell}$ is $\Delta$-exponential, so in particular every element has the same sign. Thus, $e\in E(H'_{\ell})$ is determined by the color of $e$ in $K_{\ell}$.
		
		Now let us apply Lemma \ref{lemma:domination} to $L_1,\dots,L_k$ and $h_1,\dots,h_k$. Then there exists $C\subset [N_1]$ such that $C$ is a monochromatic clique in $K_{\ell}$ for every $\ell\in[k]$ with $|C|\geq c_1N_1^{\beta}$ where $\beta=\frac{1}{rk-k+1}$ and $c_1=c_1(r,k)>0$. But then  $C$ is either a clique or an independent set in $H'_{\ell}$, which implies that $C$ is either a clique or an independent set in any Boolean combination of $H_1',\dots,H_{k}'$, so in particular, in $H'$. Here,
		$$|C|\geq c_1N_1^{\beta} = c_1\left\lfloor\frac{N_0}{z}\right\rfloor^\beta \geq c_1\left(\frac{N_0}{2z}\right)^{\beta}\geq \frac{c_1c_0^{\beta}(\log N)^{\beta/(rk)}}{(2z)^{\beta}}> c_2 (\log N)^{\frac{1}{rk(rk-k+1)}}$$
		for some $c_2=c_2(r,k)>0$. As this holds for every semi-linear hypergraph of complexity $(d,m)$, we derive that $R_{r}^{d,1,m}(n)\leq 2^{cn^{rk(rk-k+1)}}<2^{cn^{4r^{2}m^{2}}}$ for some $c=c(r,d,m)>0$, finishing the proof.
	\end{proof}
	
	The proof of Theorem \ref{thm:multicolor} follows the exact same lines, so let us give only a brief sketch.
	
	\begin{proof}[Proof of Theorem \ref{thm:multicolor}]
		Let $H^{0}_1,\dots,H^{0}_{p}$ be the colorclasses of $H$, then $H^{0}_{\ell}$ is semi-linear of complexity $(d,m)$ for $\ell\in [p]$. Therefore, $H^{0}_{\ell}$ is the Boolean combination of $2m$ primitive semi-linear hypergraphs by Lemma \ref{lemma:primitive}, let $H_{2m(\ell-1)+1},\dots,H_{2m\ell}$ be such hypergraphs. Let $k=2mp$, then there exists $C\subset V(H)$ such that $|C|\geq (\log N)^{c}$ for some $c=c(r,k)>0$ and $C$ is either a clique or an independent set in $H_{\ell}$ for $\ell\in [k]$, by the same argument as in the proof of Theorem \ref{thm:linear}. But then $C$ is a monochromatic clique in $H$, finishing the proof.
	\end{proof}
	
	\subsection{Constructions --- Proof of Theorem \ref{thm:lin_constr}}
	
	
	
		
	
	In this section, we prove Theorem \ref{thm:lin_constr}. Clearly, it is enough to prove the statement in case $r$ is even, so let us write $r=2q+2$, where $q\geq 1$ is an integer.
	
	We describe a family of $r$-uniform hypergraphs, whose members depend on the real parameters $s_1>\dots>s_q\geq 10^{6}$.
	Set $t_0=1$ and $t_{i}=s_i t_{i-1}$ for $i=1,\dots,q$. For $\ell=0,\dots,q$, define the linear functions $f_{\ell}:\mathbb{R}^{2\ell+2}  \mapsto \mathbb{R}$ such that 
	$$f_{\ell}(x_0,\dots,x_{\ell},y_{\ell},\dots,y_0)=\sum_{i=0}^{\ell}(-1)^{i}t_i(y_i-x_i).$$
	We define the $r$-uniform hypergraph $H=H(s_1,\dots,s_q)$ as follows. 
	The vertex set of $H$ is $[N]$, where we assume $N> s_1$, and given an $r$-tuple $e$ of increasing integers in $[N]$, $e$ is an edge of $H$ if $f_q(e)>0$. Clearly, $H$ is semi-linear of complexity $(1,1)$.
	
	Let us bound the clique and independence number of $H$. 
	\begin{lemma}\label{lemma:clique}
		$$\omega(H)\leq 10\left(s_q+2+\frac{\log N}{\log s_1}+\sum_{i=1}^{\lfloor (q-1)/2\rfloor }\frac{\log s_{2i}}{\log s_{2i+1}}\right).$$
	\end{lemma}
	
	\begin{proof}
		Let $p=\lfloor (q-1)/2\rfloor$, and assume that $n=\omega(H)> 10\left(s_q+2+\frac{\log N}{\log s_1}+\sum_{i=1}^{p}\frac{\log s_{2i}}{\log s_{2i+1}}\right).$
		
		Let $a_1<\dots<a_n$ be the vertices of a clique. 
		Furthermore, let $m=\lfloor n/2\rfloor-s_q-1$, and let $b_{i}=a_{n+1-i}-a_i$ for $i\in [m]$. 
		Then $b_1>\dots>b_m\ge 1$. 
		Given $1\leq j_0<\dots<j_q\leq m$, the condition that $\{a_{j_0},\dots,a_{j_q},a_{n+1-j_q},\dots,a_{n+1-j_0}\}$ is an edge of $H$ is equivalent to the inequality
		\begin{equation}\label{equ:clique}
			\sum_{i=0}^{q}(-1)^{i}t_ib_{j_i} >0.
		\end{equation}
		Let $k_0\geq 2$ be the smallest index such that $b_{k_0-1}<\frac{s_1}{2} b_{k_0}$. 
		Then for $1<i<k_0$, we have $b_{i-1}\geq \frac{s_1}{2}b_{i}$, which implies $(\frac{s_1}{2})^{k_0-2}\leq b_{k_0-1}(\frac{s_1}{2})^{k_0-2}\leq b_1< N$, hence $k_0\leq 2+\frac{\log N}{\log (s_1/2)} < 5\frac{\log N}{\log s_1}$. 
		Now for $j=1,\dots,p$, define $k_j$ as follows. If $k_{j-1}$ is already defined, let $k_j>k_{j-1}+1$ be the smallest index such that either
		\begin{enumerate}
			\item[(a)] $b_{k_{j}-1}<\frac{s_{2j+1}}{2}b_{k_j}$ or
			\item[(b)] $2s_{2j}b_{k_j-1}<b_{k_{j-1}}$.
		\end{enumerate}   
		Then for $k_{j-1}+1 < i<k_j$, we have by (a) that $b_{i}(\frac{s_{2j+1}}{2})^{i-k_{j-1}-1}\leq b_{k_{j-1}+1} < b_{k_{j-1}}$. 
		Hence, by (b), we deduce that $(\frac{s_{2j+1}}{2})^{k_j-k_{j-1}-3}<2s_{2j}$, which implies $k_j-k_{j-1}\leq 3+\frac{\log (2s_{2j})}{\log (s_{2j+1}/2)}<5\frac{\log s_{2j}}{\log s_{2j+1}}$. Then $$k_j<5\,\frac{\log N}{\log s_1}+5\sum_{i=1}^{j}\frac{\log s_{2i}}{\log s_{2i+1}}<m,$$ which shows that the indices $k_0,\dots,k_p$ are well defined. Next, we prove by induction that for $\ell=0,\dots,p$, we have
        \begin{equation}\label{eq: telescope}
          \sum_{i=0}^{\ell}(t_{2i}b_{k_i-1}-t_{2i+1}b_{k_i}) < -\frac{t_{2\ell+1}}{2}b_{k_{\ell}}.
        \end{equation}
		In the base case $\ell=0$, the left hand side is $b_{k_0-1}-t_1 b_{k_0}$, so noting that $t_1=s_1$ and using the inequality $b_{k_0-1}<\frac{s_1}{2}b_{k_0}$, the  desired inequality
		$$b_{k_0-1}-t_1 b_{k_0}<-\frac{t_1}{2}b_{k_0}$$
		is satisfied. Now let us assume $\ell>0$, and consider two cases.
		\begin{description}
			\item[Case 1.] $b_{k_{\ell}-1}<\frac{s_{2\ell+1}}{2}b_{k_{\ell}}$.
			
			By our induction hypothesis, we can upper bound the contribution of the first $\ell$ terms by 0, so we can write
			$$\sum_{i=0}^{\ell}(t_{2i}b_{k_i-1}-t_{2i+1}b_{k_i}) < t_{2\ell}b_{k_\ell-1}-t_{2\ell+1}b_{k_\ell}<-\frac{t_{2\ell+1}}{2}b_{k_{\ell}}.$$
			Here, the last inequality holds by the fact that $t_{2\ell+1}=s_{2\ell+1}t_{2\ell}>0    $.
			\item[Case 2.] $2s_{2\ell}b_{k_\ell-1}<b_{k_{\ell-1}}$.
			
			By our induction hypothesis, we can upper bound the contribution of the first $\ell$ terms by $-\frac{t_{2\ell-1}}{2}b_{k_{\ell-1}}$, so we get
			$$\sum_{i=0}^{\ell}(t_{2i}b_{k_i-1}-t_{2i+1}b_{k_i}) < -\frac{t_{2\ell-1}}{2}b_{k_{\ell-1}}+t_{2\ell}b_{k_\ell-1}-t_{2\ell+1}b_{k_\ell}<-t_{2\ell+1}b_{k_{\ell}}<-\frac{t_{2\ell+1}}{2}b_{k_{\ell}}.$$
			Here, the second inequality holds by the fact that $s_{2\ell}t_{2\ell-1}=t_{2\ell}>0$.
		\end{description}
		If $q$ is odd, (\ref{eq: telescope}) with $\ell=p$ shows that the set $\{b_{k_j-1},b_{k_j}:j=0,\dots,p\}$ violates (\ref{equ:clique}), a contradiction. 
		So suppose that $q$ is even. Writing $e_0$ for the $(r-2)$-tuple we get by increasingly ordering the set $ \{a_{k_i-1},a_{k_i},a_{n+1-k_i},a_{n+2-k_i}:i=0,\dots,p\}$, (\ref{eq: telescope}) with $\ell=p$ implies
		$$f_{q-1}(e_0)=\sum_{i=0}^{p}(t_{2i}b_{k_i-1}-t_{2i+1}b_{k_i})<-\frac{t_{q-1}}{2}b_{k_{p}}=-\frac{t_{q-1}}{2}(a_{n+1-k_{p}}-a_{k_p}).$$
		As $k_p<m\leq\frac{n}{2}-s_q-1$, there are at least $2s_q+1$ integers between $k_p+1$ and $n-k_p$. 
		Hence, there is an index $b \in \{k_p+1,k_p+2,\dots,n-k_p-1\}$ such that $(a_{b+1}-a_b)\leq \frac{a_{n-k_p}-a_{k_p+1}}{2s_q} < \frac{a_{n+1-k_p}-a_{k_p}}{2s_q}$, thereby $t_q(a_{b+1}-a_b) < \frac{t_{q-1}}{2}(a_{n+1-k_{p}}-a_{k_p})$. Let $e$ be the $r$-tuple we get after inserting $a_b$ and $a_{b+1}$ into $e_0$ respecting the increasing order. Then
		$$f_q(e)=f_{q-1}(e_0)+t_q(a_{b+1}-a_b)<-\frac{t_{q-1}}{2}(a_{n+1-k_{p}}-a_{k_p})+t_q(a_{b+1}-a_b)<0,$$
		contradicting that $e$ is an edge of $H$.
	\end{proof}
	
	\begin{lemma}\label{lemma:independent}
		$$\alpha(G)\leq  10\left(s_q+2+\sum_{i=1}^{\lfloor q/2\rfloor}\frac{\log s_{2i-1}}{\log s_{2i}}\right).$$
	\end{lemma}	
	
	\begin{proof}
	The proof of this is very similar to the proof of Lemma \ref{lemma:clique}, so we only give an outline. 
	Let $p=\lfloor q/2\rfloor$, and assume that $n=\alpha(H)> 10\left(s_q+2+\sum_{i=1}^{p}\frac{\log s_{2i-1}}{\log s_{2i}}\right).$
		
	Let  $a_1<\dots<a_n$ be the vertices of an independent set. Furthermore, let $m=\lfloor n/2\rfloor-s_q-1$, let $b_{i}=a_{n+1-i}-a_i$ for $i\in [m]$. Then $b_1>\dots>b_m$. Given $1\leq j_0<\dots<j_q\leq m$, the condition that $\{a_{j_0},\dots,a_{j_q},a_{n+1-j_q},\dots,a_{n+1-j_0}\}$ is not an edge of $H$ is equivalent to the inequality
		\begin{equation}\label{equ:independent}
			\sum_{i=0}^{q}(-1)^{i}t_ib_{j_i} \leq 0.
		\end{equation}
	Define the indices $1=k_0<k_1<\dots<k_p<m$ as follows. If $k_{j-1}$ is already defined, let $k_j>k_{j-1}+1$ be the smallest index such that either
		\begin{enumerate}
			\item[(a)] $b_{k_{j}-1}<\frac{s_{2j}}{2}b_{k_j}$ or
			\item[(b)] $2s_{2j-1}b_{k_j-1}<b_{k_{j-1}}$.
		\end{enumerate}  
	Then $k_{j}-k_{j-1} < 5\frac{\log s_{2j-1}}{\log s_{2j}}$, so $k_{j}< 1+ 5\sum_{i=1}^{j}\frac{\log s_{2i-1}}{\log s_{2i}}<m$, showing that $k_1,\dots,k_p$ are well defined. We prove by induction that for $\ell=0,\dots,p$, we have
    \begin{equation}\label{eq: telescope_2}
      b_{1}-\sum_{i=1}^{\ell}(t_{2i-1}b_{k_i-1}-t_{2i}b_{k_i}) >  \frac{t_{2\ell}}{2}b_{k_\ell}.
    \end{equation}
    This is trivially true for $\ell=0$, so let us assume that $\ell\geq 1$. Consider two cases.
		\begin{description}
			\item[Case 1.] $b_{k_{\ell}-1}<\frac{s_{2\ell}}{2}b_{k_{\ell}}$.
			
			By our induction hypothesis, we have $b_1-\sum_{i=1}^{\ell-1}(t_{2i-1}b_{k_i-1}-t_{2i}b_{k_i})>0$, so we can write
			$$b_1-\sum_{i=1}^{\ell}(t_{2i-1}b_{k_i-1}-t_{2i}b_{k_i})> -t_{2\ell-1}b_{k_\ell-1}+t_{2\ell}b_{k_\ell}>\frac{t_{2\ell}}{2}b_{k_{\ell}}.$$
			\item[Case 2.] $2s_{2\ell-1}b_{k_\ell-1}<b_{k_{\ell-1}}$.
			
			By our induction hypothesis, we have
				$$b_1-\sum_{i=1}^{\ell}(t_{2i-1}b_{k_i-1}-t_{2i}b_{k_i})> \frac{t_{2\ell-2}}{2}b_{k_{\ell-1}}-t_{2\ell-1}b_{k_\ell-1}+t_{2\ell}b_{k_\ell}>t_{2\ell}b_{k_{\ell}}>\frac{t_{2\ell}}{2}b_{k_{\ell}}.$$
		\end{description}
		If $q$ is even, (\ref{eq: telescope_2}) with $\ell=p$ shows that the set $\{b_1\}\cup\{b_{k_j-1},b_{k_j}:j=1,\dots,p\}$ violates (\ref{equ:independent}), a contradiction. 
		So suppose that $q$ is odd. Writing $e_0$ for the $(r-2)$-tuple we get by increasingly ordering the set $\{a_1,a_n\}\cup \{a_{k_i-1},a_{k_i},a_{n+1-k_i},a_{n+2-k_i}:i=1,\dots,p\}$, (\ref{eq: telescope_2}) with $\ell=p$ implies
		$$f_{q-1}(e_0)=b_1-\sum_{i=1}^{p}(t_{2i}b_{k_i-1}-t_{2i+1}b_{k_i})>\frac{t_{q-1}}{2}b_{k_{p}}=\frac{t_{q-1}}{2}(a_{n+1-k_{p}}-a_{k_p}).$$
		As $k_p<m\le\frac{n}{2}-s_q-1$, there are at least $2s_q+1$ integers between $k_p+1$ and $n-k_p$. 
		Hence, there is an index $b \in \{k_p+1,k_p+2,\dots,n-k_p-1\}$ such that $(a_{b+1}-a_b)\leq \frac{a_{n-k_p}-a_{k_p+1}}{2s_q} < \frac{a_{n+1-k_p}-a_{k_p}}{2s_q}$, thereby $t_q(a_{b+1}-a_b) < \frac{t_{q-1}}{2}(a_{n+1-k_{p}}-a_{k_p})$. Let $e$ be the $r$-tuple we get after inserting $a_b$ and $a_{b+1}$ into $e_0$ respecting the increasing order. Then
		$$f_q(e)=f_{q-1}(e_0)-t_q(a_{b+1}-a_b)>\frac{t_{q-1}}{2}(a_{n+1-k_{p}}-a_{k_p})-t_q(a_{b+1}-a_b)>0,$$
		contradicting that $e$ is not an edge of $H$.
	\end{proof}
	
	We are now ready to prove the main theorem of this section.
	
	\begin{proof}[Proof of Theorem \ref{thm:lin_constr}]
	    We assume that $N$ is sufficiently large with respect $r$. For $i=1,\dots,q-1$, set $s_{i}=2^{(\log N)^{(q-i)/q}}$, and set $s_q=10^{6}$. Then $N > s_1$ is satisfied assuming $N$ is large. Let $H=H(s_1,\dots,s_q)$ be the semi-linear hypergraph of complexity $(1,1)$ defined as above. Then by Lemma \ref{lemma:clique}, we have
	    $$\omega(H)\leq 10\left(s_q+2+\frac{\log N}{\log s_1}+\sum_{i=1}^{\lfloor (q-1)/2\rfloor }\frac{\log s_{2i}}{\log s_{2i+1}}\right)<10^{8}+10q(\log N)^{1/q},$$
	    and by Lemma \ref{lemma:independent}, we can write
	    $$\alpha(H)\leq  10\left(s_q+2+\sum_{i=1}^{\lfloor q/2\rfloor}\frac{\log s_{2i-1}}{\log s_{2i}}\right)< 10^{8}+10q(\log N)^{1/q}.$$
	    By writing $n=10^{8}+10q(\log N)^{1/q}$, this shows that $R^{1,1,1}_{r}(n)>N>2^{cn^{q}}$, where $c=c(r)>0$ is some suitable constant.
	\end{proof}
		
	\section{Concluding remarks}\label{sect:concluding remarks}
	
	\begin{itemize}
		
		\item	Theorem \ref{thm:construction} provides new lower bounds on the asymmetric semi-algebraic Ramsey number $R_{3}^{\mathbf{t}}(s,n)$. However, there is still a large gap between this lower bound and the upper bound by Suk \cite{Suk16}. It would be interesting to close this gap, and we believe our lower bound should be closer to the truth.
		
		\begin{conjecture}
			Let $s,d,D,m$  be positive integers, then there exists $c=c(s,d,D,m)>0$ such that
			$$R_3^{d,D,m}(s,n)<2^{(\log n)^c}$$
			holds for every sufficiently large $n$.
		\end{conjecture}
		
		\item		In the case of graphs, it was observed by Tomon \cite{Tlin} that the semi-linear Ramsey number satisfies $R_2^{d,1,m}(n)>n^{c}$, assuming $d$ and $m$ are sufficiently large with respect to $c$. In particular, one can define graphs achieving such bounds using \emph{zero-patterns} of linear functions instead of sign-patterns, i.e. whether two vertices are connected depends only on which linear functions evaluate to 0. However, for higher uniformity, it was proved by Sudakov and Tomon~\cite{SudTom21} that $r$-uniform hypergraphs defined by zero-patterns of polynomials (over any field $\mathbb{F}$) contain polynomial sized cliques or independent sets. 
		
		\item Theorem \ref{thm:lin_constr} shows that $R^{1,1,1}_r(n)\geq 2^{\Omega(n^{\lfloor r/2\rfloor-1})}$. On the other hand, a slight refinement of our proof gives the upper bound $R^{1,1,1}_r(n)\leq 2^{O(n^{r})}$. Indeed, let $H$ be an $r$-uniform semi-linear hypergraph of complexity $(1,1)$ on $N$ vertices. Then the edges of $H$ (or its complement) are either defined by the zeros of a single linear function, or they are defined by a single linear inequality. In the former case, $H$ contains polynomial sized cliques or independent sets by the aforementioned result of Sudakov and Tomon \cite{SudTom21}. However, in the latter case, $H$ is also primitive semi-linear, and its witness $P\in\mathbb{R}^{r\times N}$ has the property that any two of its rows are constant multiples of each other. Hence, $P$ contains a cupcap submatrix of size $r\times \Omega_r(\log N)$. Repeating the rest of our proof implies the desired bound  $R^{1,1,1}_r(n)\leq 2^{O(n^{r})}$.
		
		\item There is another approach of defining semi-algebraic Ramsey numbers, following \cite{BM12}. So far, we viewed the complexity $\mathbf{t}=(d,D,m)$ as fixed, while the defining polynomials may depend on $n$. 
		On the other hand, one can also fix the description of the semi-algebraic hypergraph, that is, fix a collection of $m$ polynomials $f_1\dots,f_m$, and a function $\Phi:\{+,-,0\}^m\mapsto \{\mbox{True},\mbox{False}\}$. We call $\Lambda=(f_1,\dots,f_m; \Phi)$ a \emph{description}. 
		Define the Ramsey number $R_r^{\Lambda}(s,n)$ to be the smallest $N$ such that any $r$-uniform semi-algebraic hypergraph on $N$ vertices with description $\Lambda$ contains a clique of size $s$ or an independent set of size $n$, and set $R_r^{\Lambda}(n)=R_r^{\Lambda}(n,n)$. It is not difficult to show the following.
		
		\begin{theorem}
			Let $\Lambda$ be a description of complexity $(1,1,1)$. Then there exists a constant $c=c(\Lambda)>0$ such that 
			$$R^{\Lambda}_r(n)\leq 2^{cn}.$$
		\end{theorem}
		
		Indeed, given an $r$-uniform hypergraph $H$ on $N$ vertices with description $\Lambda=(f;\Phi)$, the vertex set of $H$ corresponds to a sequence of real numbers $x_1,\dots,x_N$. Choose $\Delta=2r\max_{a,b}|\frac{a}{b}|$, where the maximum ranges over all non-zero coefficients of the linear function~$f$. Then by Lemma \ref{lemma:convex} and Lemma \ref{lemma:exponential}, there exists $N_0=\Omega_{\Delta}(\log N)$ and a subsequence $\mathbf{y}=(y_1,\dots,y_{N_0})$ of $x_1,\dots,x_{N}$ such that some shift of $\mathbf{y}$ is $\Delta$-exponential. It is easy to show that the subhypergraph of $H$ induced on the set of vertices corresponding to $\mathbf{y}$ contains either a clique or an independent set of size $\Omega(N_0)=\Omega_{\Delta}(\log N)$.
		
		On the other hand, there exists a description $\Lambda$ of complexity $(r,1,1)$ such that $R^{\Lambda}_r(n)\geq 2^{\Omega(n^{\lfloor r/2\rfloor-1})}.$ 
		Indeed, take $\Lambda$ to be the description of the primitive $r$-uniform semi-linear hypergraph, and note that the construction in Theorem \ref{thm:lin_constr} has primitive complexity 1.
		
		\item	The result of Bukh and Matou\v sek \cite{BM12} shows that if $d=1$, then $R_{r}^{d,D,m}(n)<2^{2^{O(n)}}$, while we proved that if $D=1$, then $R_r^{d,D,m}(n)<2^{n^{O(1)}}$. It seems already highly challenging to decide whether in case $d=2$ or $D=2$, the height of the tower in the general upper bound $R_r^{d,D,m}(n)<\tower_{r-1}(n^{O(1)})$ can be reduced. Nevertheless, we propose the following conjecture.
		
		\begin{conjecture}
			There exists a function $k:\mathbb{N}\rightarrow\mathbb{N}$ such that the following holds. Let $r,d,D,m$ be integers, and let $s=\min\{d,D\}$. Then there exists $c=c(r,d,D,m)$ such that
			$$R_r^{d,D,m}(n)< \mbox{tw}_{k(s)}(n^c).$$
		\end{conjecture}
		
		Note that the previous results imply that one can take $k(1)=3$. However, as remarked earlier, proving the existence of $k(2)$ looks rather difficult already.

		\item Finally, we recall the problem about finding cupcap submatrices of large size, that is, determining the optimal version of Lemma \ref{lemma:convex}. 
		
		\begin{problem}
			Given positive integer $q$, determine the order of magnitude of the function~$f_q(N)$, denoting the largest $N'$ such that any matrix $M\in \mathbb{R}^{q\times N}$ with no repeated elements contains a $q\times N'$-sized cupcap submatrix.
		\end{problem}
		
	\end{itemize}
	
	\vspace{0.3cm}
	\noindent	
	{\bf Acknowledgements.}
	We would like to thank the anonymous referees for their useful comments and suggestions.

Zhihan Jin is supported by the SNSF grant 200021\_196965. Part of this research was done while Istv\'an Tomon was employed at ETH Zurich, where he was supported by the SNSF grant 200021\_196965.

\end{document}